\documentclass[11pt,reqno]{amsart} 
 \usepackage{amsmath,amssymb,amsthm}
 \usepackage{graphicx}
 \usepackage{cite}
 \usepackage[mathscr]{eucal}

 \theoremstyle{plain}
 \newtheorem{theorem}{Theorem}[section]  
 \newtheorem{corollary}[theorem]{Corollary}  
 \newtheorem{lemma}[theorem]{Lemma}  
 
 \newtheorem*{theorem*}{Theorem}

 \theoremstyle{plain}  
 
 \newtheorem{definition}[theorem]{Definition}  
   
 \newtheorem*{example*}{Example}

 \newtheoremstyle{citing}
   {3pt}
   {3pt}
   {\itshape}
   {}
   {\bfseries}
   {.}
   {.5em}
   {\thmnote{#3}}

 \theoremstyle{citing}

\numberwithin{equation}{section}
\allowdisplaybreaks[1]

\newlength{\intwidth}
\makeatletter
\DeclareRobustCommand{\cpvint}[2]
    {\mathop{%
       \text{%
         \settowidth{\intwidth}{%
           \ifx\ilimits@\displaylimits
             $\int_{#1}^{#2}$%
           \else
             $\int$%
           \fi}%
         \makebox[0pt][l]{\makebox[\intwidth]{$\text{C}$}}%
         $\int_{#1}^{#2}$}}}
\makeatother

\makeatletter
\DeclareRobustCommand{\cpvintsmall}[2]
    {\mathop{%
       \text{%
         \settowidth{\intwidth}{%
           \ifx\ilimits@\displaylimits
             $\int_{#1}^{#2}$%
           \else
             $\int$%
           \fi}%
         \makebox[0pt][l]{\makebox[\intwidth]{$\text{{\tiny C}}$}}%
         $\int_{#1}^{#2}$}}}
\makeatother

\newcommand{\hz}{\mathbb{H}}

\newcommand{\xpfeil}{\xrightarrow}  
  
\newcommand{\rand}{\partial} 
\newcommand{\where}{:\:}  
\newcommand{\iso}{\cong}  
  
\newcommand{\sgn}{\text{sgn}}

\newcommand{\cupl}{\mathop{\cup}\limits}  
\newcommand{\capl}{\mathop{\cap}\limits}

\newcommand{\aequi}{\Longleftrightarrow}  
  
\newcommand{\laplace}{\Delta}

\newcommand{\nz}{{\mathbb N}}

\newcommand{\rz}{{\mathbb R}}  
\newcommand{\zz}{{\mathbb Z}}  
 
\newcommand{\eps}{\varepsilon}  
\renewcommand{\phi}{\varphi} 
\newcommand{\eval}{\vert}
\newcommand{\cross}{\times}

\begin{document}
 
\title[Closed magnetic geodesics] {Closed magnetic geodesics on closed hyperbolic Riemann surfaces}
\author{Matthias Schneider}
\address{Ruprecht-Karls-Universit\"at\\
         Im Neuenheimer Feld 288\\
         69120 Heidelberg, Germany\\}
\email{mschneid@mathi.uni-heidelberg.de} 
\date{September 9, 2010}  

\begin{abstract}
We prove the existence of Alexandrov embedded closed magnetic geodesics on 
closed hyperbolic surfaces. Closed magnetic geodesics correspond to closed
curves with prescribed geodesic curvature.

\end{abstract}

\maketitle

\section{Introduction}
\label{sec:introduction}
Let $(M,g)$ be a compact, two dimensional, oriented manifold 
equipped with a smooth metric $g$ and
$k:M \to \rz$ a smooth positive function. We consider the following two equations
for curves $\gamma$ on $M$:
\begin{align}
\label{eq:magnetic_geodesic}
D_{t,g} \dot \gamma = k(\gamma) J_{g}(\gamma) \dot \gamma,
\end{align}
and
\begin{align}
\label{eq:1}
D_{t,g} \dot \gamma = |\dot \gamma|_{g} k(\gamma) J_{g}(\gamma) \dot \gamma,
\end{align}
where $D_{t,g}$ is the covariant derivative with respect to $g$,
and $J_g(x)$ is the rotation by $\pi/2$ in $T_xM$ measured with $g$ and 
the given orientation.\\
Equation (\ref{eq:magnetic_geodesic}) describes the motion
of a charge in a magnetic field corresponding to the magnetic form
$k dV_g$
and solutions to \eqref{eq:magnetic_geodesic}
will be called {\em ($k$-)magnetic geodesics}
(see \cite{MR890489,MR676612}).
Equation (\ref{eq:1}) corresponds to the problem of {\em prescribing geodesic curvature}, 
as its solutions $\gamma$ 
are constant speed curves with geodesic curvature $k_g(\gamma,t)$ given by $k(\gamma(t))$
(see \cite{arXiv:0808.4038}).\\
It is easy to see that a nonconstant magnetic geodesic $\gamma$ lies in a fixed energy level 
$E_c$, i.e. there is $c>0$, such that
\begin{align*}
(\gamma,\dot\gamma) \in E_c := \{(x,V) \in TM\where |V|_g=c\}.  
\end{align*}
For fixed $k$ and $c>0$ 
the equations \eqref{eq:magnetic_geodesic} and \eqref{eq:1} are equivalent in the following
sense: If $\gamma$ is a nonconstant solution of \eqref{eq:1} with $k$ replaced by
$k/c$, then the curve 
$\gamma_c(t):=\gamma(ct/|\dot \gamma|_g)$
is a $k$-magnetic geodesic in $E_c$, and a $k$-magnetic geodesic in $E_c$
solves \eqref{eq:1} with $k$ replaced by $k/c$. We emphazise that $k$-magnetic geodesics
on different energy levels are not reparameterizations of each other.\\
We study the existence of closed curves with prescribed geodesic curvature 
or equivalently the existence of periodic magnetic geodesics on prescribed energy levels 
$E_c$.\\
There is a vast literature on the existence of closed magnetic geodesics.
We limit ourselves to quote \cite{MR730159,MR1185286}
for the approach via Morse-Novikov theory
for (possibly multi-valued) variational functionals, 
\cite{MR890489,MR902290} for the
application of the theory of dynamical
systems and symplectic geometry, \cite{MR2036336} concerning Aubry-Mather's theory,
and \cite{arXiv:0808.4038}, where the theory 
of vector fields on infinite dimensional manifolds is applied to \eqref{eq:1}.
We refer to \cite{MR2036336,MR2534483,MR1432462,MR2593232} for a survey 
and additional references.\\
From the example of the horocycle flow below, closed magnetic geodesics need not
exist on a fixed energy level in general. However, from \cite{MR1185286,MR1133303,MR1432462}, 
there are always closed magnetic geodesics
in high and low energy levels, i.e. in $E_c$ with $c\ge c_0$ and $c\le (c_0)^{-1}$, 
where $c_0>0$ depends on $(M,g)$ and $k$ (in case of a flat torus and high energy levels 
$k$ is assumed
not to vanish).
If the magnetic form is exact, i.e. $[k dV_g]=0$ in $H_{dR}^2(M)$,
then there is a periodic magnetic geodesic in every energy level (see \cite{MR2036336}).
Concerning non exact magnetic forms positive functions $k$ are of special
interest, since the magnetic form is symplectic in this case. 
For $k>0$ a closed magnetic geodesic exists
in every energy level, if $(M,g)$ is a flat torus \cite{MR970068,MR786086}
or if $(M,g)$ is a sphere $S^2$ with nonnegative curvature
\cite{arXiv:0903.1128}. The (essentially) only nonexistence result 
for closed magnetic geodesics is based on an old result of
Hedlund \cite{MR1545946}. 
\begin{example*}[Horocycle flow \cite{MR1432462}]
Let $(M,g)$ be a compact hyperbolic surface of constant curvature $K_g\equiv -1$
and $k \equiv 1$.
\begin{enumerate}
\item  If $0<c<1$, then $E_c$ contains a contractible closed magnetic geodesic.
\item  There are no closed magnetic geodesics in $E_1$.
\item If $c>1$, there are no contractible closed magnetic geodesics in $E_c$,
but any non trivial free homotopy class of closed curves can be represented by one.
\end{enumerate}
\end{example*}
The existence question for closed magnetic geodesics on hyperbolic surfaces
for non constant functions $k$ is poorly understood. We shall show: If 
$(M,g)$ is a compact hyperbolic surface with Gaussian curvature $K_g \ge -1$
and $k\ge 1$ a positive function, then there is a 
contractible closed magnetic geodesics in $E_c$ for all $0<c<1$.
The example of the horocycle flow shows that this existence result is sharp.\\
We consider curves, that
are Alexandrov embedded.
\begin{definition}{(oriented Alexandrov embedded)}
\label{def:alexandrov}
Let $B \subset \rz^2$ denote the open ball of radius $1$ centered at $0\in \rz^2$.
An immersion $\gamma\in C^1(\rand B,M)$ will be called {\em oriented Alexandrov embedded},
if there is an immersion $F\in C^1(\overline{B},M)$, such that
$F\eval_{\rand B}=\gamma$ and $F$ is orientation preserving in the sense that
\begin{align*}
\langle DF\eval_{x}x, J_g(\gamma(x))\dot\gamma(x)\rangle_{T_{\gamma(x)}S^2,g}>0
\end{align*}
for all $x \in \rand B$.
\end{definition}
We shall prove
\begin{theorem}
\label{thm_existence}
Let $(M,g)$ be a smooth, compact, orientable surface with negative Euler characteristic 
and $k \in C^\infty(M)$ a positive function. Assume there is $K_0>0$ such that
$k$ and the Gaussian curvature $K_g$ of $(M,g)$ satisfy
\begin{align*}
k>(K_0)^\frac12 \text{ and } K_g\ge -K_0.  
\end{align*}
Then there is an oriented Alexandrov embedded curve $\gamma \in C^2(S^1,M)$ that
solves \eqref{eq:1} and
the number of such solutions is at least
$-\chi(M)$
provided they are all nondegenerate.
\end{theorem}
The equivalence between (\ref{eq:magnetic_geodesic}) and (\ref{eq:1}) leads to
\begin{corollary}
\label{cor_existence}
Let $(M,g)$ be a smooth, compact, orientable surface with negative Euler characteristic 
and $k \in C^\infty(M)$ a positive function. Assume there is $K_0>0$ such that
$k$ and the Gaussian curvature $K_g$ of $(M,g)$ satisfy
\begin{align*}
k\ge(K_0)^\frac12 \text{ and } K_g\ge -K_0.  
\end{align*}
Then every energy level $E_c$ with $0<c<1$
contains an oriented Alexandrov embedded closed magnetic geodesic
and the number of such closed magnetic geodesics in $E_c$ is at least $-\chi(M)$
provided they are all nondegenerate.
\end{corollary}
The proof of our existence results is organized as follows.
We consider solutions to \eqref{eq:1} as zeros of 
the vector field $X_{k,g}$ defined on the Sobolev space $H^{2,2}(S^1,M)$:
For $\gamma \in H^{2,2}(S^1,M)$ we let $X_{k,g}(\gamma)$ be the unique weak solution of   
\begin{align}
\label{eq:def_vector_field}
\big(-D_{t,g}^{2} + 1\big)X_{k,g}(\gamma)= 
-D_{t,g} \dot\gamma + |\dot \gamma|_{g} k(\gamma)J_{g}(\gamma)\dot\gamma   
\end{align}
in $T_\gamma H^{2,2}(S^1,M)$.  
The uniqueness implies that any zero of $X_{k,g}$ is a weak solution of \eqref{eq:1}
which is a classical solution in $C^{2}(S^1,M)$ applying standard regularity
theory.\\ 
After setting up notation in Section \ref{s:preliminiaries} and introducing the classes of maps
and spaces needed for our analysis 
we recall in Section \ref{sec:s1_euler}
the definition and properties of the $S^1$-equivariant Poincar\'{e}-Hopf index
defined in \cite{arXiv:0808.4038}, 
$$\chi_{S^1}(X_{k,g},\mathscr{M}_A) \in \zz,$$ 
where
$\mathscr{M}_A$ is the set of oriented Alexandrov embedded regular curves
in $H^{2,2}(S^1,M)$.\\
From the uniformization theorem 
$(M,g)$
is isometric to $(\hz/\Gamma,e^\phi g_0)$,
where $\Gamma$ is a group
of isometries of the standard hyperbolic plane $(\hz,g_0)$ acting
freely and properly discontinuously and $\phi$ is a function in
$C^\infty(\hz/\Gamma,\rz)$.
Since the problem of prescribing
geodesic curvature is invariant under isometries we may assume
without loss of generality that
\begin{align*}
(M,g) = (\hz/\Gamma,e^\phi g_0). 
\end{align*}
In Section \ref{sec:unperturbed} we analyze the {\em unperturbed problem} 
with $k\equiv k_0>0$ and $g=g_0$:
We compute the set of oriented Alexandrov embedded zeros of $X_{k_0,g_0}$
and the image and kernel of the corresponding linearizations.
The perturbative analysis in Section \ref{sec:perturbed}, which carries
over from \cite{arXiv:0808.4038}, is used to compute the degree
of the unperturbed problem in Section \ref{sec:computation_degree}:   
For large positive constants $k_0$ and the standard metric $g_0$
we shall show that
\begin{align*}
\chi_{S^1}(X_{k_0,g_0},\mathscr{M}_A)=-\chi(M),  
\end{align*}
where $\chi(M)$ denotes the Euler characteristic of $M$.\\
Section \ref{sec:apriori-estimate} contains the apriori estimate which implies
that under the assumptions of Theorem \ref{thm_existence}
the set of solutions to \eqref{eq:1} is compact in $\mathscr{M}_A$.
The homotopy invariance of the $S^{1}$-equivariant Poincar\'{e}-Hopf index then leads
to the identity
\begin{align*}
\chi_{S^1}(X_{k,g},\mathscr{M}_A)=\chi_{S^1}(X_{k_0,g_{can}},\mathscr{M}_A)=-\chi(M).  
\end{align*}
The resulting proof of Theorem \ref{thm_existence} is given in Section \ref{sec:existence}.

\section{Preliminaries}
\label{s:preliminiaries}
It is convenient for the functional analytic setting to assume that $M$ is embedded
in some $\rz^{q_M}$.
We consider for $m\in \nz_0$ the set of Sobolev functions
\begin{align*}
H^{m,2}(S^{1},M) := \{\gamma \in H^{m,2}(S^{1},\rz^{q_M}) \where \gamma(t) \in M 
\text{ for a.e. } t \in S^{1}.\}   
\end{align*}
For $m \ge 1$ the set
$H^{m,2}(S^{1},M)$ is a sub-manifold of the Hilbert space $H^{m,2}(S^{1},\rz^{q_M})$
and is contained in $C^{m-1}(S^{1},\rz^{q_M})$.
Hence, if $m\ge 1$ then $\gamma \in H^{m,2}(S^{1}, M)$ 
satisfies 
$\gamma(t) \in M$ for all $t\in S^{1}$. 
In this case the tangent space 
at $\gamma \in H^{m,2}(S^{1}, M)$ is given by
\begin{align*}
T_\gamma H^{m,2}(S^{1}, M) := 
\{V \in H^{m,2}(S^{1},\rz^{q_M})\where V(t) \in T_{\gamma(t)}M \text{ for all }t \in S^{1}\}.  
\end{align*}
For $m=0$ the set $H^{0,2}(S^{1},M)=L^{2}(S^{1},M)$ fails to be a manifold. 
We define for $\gamma \in H^{1,2}(S^{1},M)$ the space 
$T_\gamma L^{2}(S^{1},M)$ by
\begin{align*}
T_\gamma L^{2}(S^{1},M):= 
\{V \in L^{2}(S^{1},\rz^{q_M})\where V(t) \in T_{\gamma(t)}M \text{ for a.e. }t \in S^{1}\}.  
\end{align*}
A metric $g$ on $M$ induces a metric on $H^{m,2}(S^{1},M)$ for $m\ge 1$ 
by setting for $\gamma \in H^{m,2}(S^{1},M)$
and $V,\,W \in T_\gamma H^{m,2}(S^{1}, S^{2})$
\begin{align*}
\langle W,V\rangle_{T_\gamma H^{m,2}(S^{1}, S^{2}),g} 
:= \int_{S^{1}}\Big\langle &\big((-1)^{\llcorner \frac{m}{2}\lrcorner}(D_{t,g})^{m}+1\big)V(t),\\
&\big((-1)^{\llcorner \frac{m}{2}\lrcorner}(D_{t,g})^{m}+1\big)W(t)\Big\rangle_{\gamma(t),g}\, dt,   
\end{align*}
where $\llcorner m/2\lrcorner$ denotes the largest integer that does not exceed $m/2$.\\
Since $g$ and $k$ are smooth, $X_{k,g}$ is a smooth vector field (see
\cite[Sec. 6]{arXiv:0808.4038,MR0464304}) 
on the set $H^{2,2}_{reg}(S^{1},M)$ of regular curves,
\begin{align*}
H^{2,2}_{reg}(S^{1},M) 
:= \{\gamma \in H^{2,2}(S^{1},M)\where \dot\gamma(t) \neq 0 \text{ for all }t\in S^{1}\}. 
\end{align*}
From \cite{arXiv:0808.4038} there holds
\begin{align}
\label{eq:dg_x_g_formula}
\big(-&D_{t,g}^2+1\big)D_{g} X_{k,g}\eval_\gamma(V) \notag\\
&= -D_{t,g}^2 V - R_g\big(V,\dot\gamma \big)\dot\gamma
+ |\dot\gamma|_g^{-1}\langle D_{t,g}V,\dot\gamma\rangle_g k(\gamma)J_{g}(\gamma)\dot\gamma
 \notag \\
&\quad +|\dot\gamma|_g \big(k'(\gamma)V\big)J_{g}(\gamma)\dot\gamma
+ |\dot\gamma|_g k(\gamma) 
\Big(\big(D_gJ_{g}\eval_{\gamma}V\big)\dot\gamma
+J_{g}(\gamma) D_{t,g}V\Big). 
\end{align}
We note that (see also \cite[Thm. 6.1]{MR493919})
\begin{align*}
\big(-D_{t,g}^2+1\big)D_{g} X_{k,g}\eval_\gamma(V) = (-D_{t,g}^2 +1)V +T(V),   
\end{align*}
where $T$ is a linear map from $T_\gamma H^{2,2}(S^{1},M)$ 
to $T_\gamma L^{2}(S^{1},M)$ that
depends only on the first derivatives of $V$ and is therefore compact.
Taking the inverse $(-D_{t,g}^2+1)^{-1}$ we deduce that $D_{g} X_{k,g}\eval_\gamma$
is the form $identity + compact$ and thus a Rothe map (see \cite{arXiv:0808.4038}).\\
The vector field $X_{k,g}$ as well as the set of solutions to \eqref{eq:1}
is invariant under a circle action:
For $\theta \in S^{1}=\rz/\zz$ and $\gamma \in H^{2,2}(S^1,M)$
we define $\theta*\gamma \in H^{2,2}(S^1,M)$
by 
\begin{align*}
\theta*\gamma(t) = \gamma(t+\theta).  
\end{align*}
Moreover, for $V \in T_\gamma H^{2,2}(S^1,M)$ we let
\begin{align*}
\theta*V := V(\cdot+\theta) \in T_{\theta*\gamma} H^{2,2}(S^1,M).  
\end{align*}
Then $X_{k,g}(\theta * \gamma) = \theta*X_{k,g}(\gamma)$ for any $\gamma \in H^{2,2}(S^1,M)$
and $\theta\in S^{1}$. Thus, any zero gives rise to a $S^{1}$-orbit of zeros.
We call $\gamma$ a {\em prime} curve, if
the isotropy group $\{\theta \in S^1\where \theta*\gamma=\gamma\}$ of $\gamma$ is trivial.\\
For $m\ge 1$ the exponential map $Exp_{g}: TH^{m,2}(S^{1},M) \to H^{m,2}(S^{1},M)$
is defined for $\gamma \in H^{m,2}(S^{1},M)$ and $V \in T_\gamma H^{m,2}(S^{1},M)$ by
\begin{align*}
Exp_{\gamma,g}(V)(t) := Exp_{\gamma(t),g}(V(t)),  
\end{align*}
where $Exp_{z,g}$ denotes the exponential map on $(M,g)$ at $z \in M$. 
Due to its pointwise definition
\begin{align*}
\theta*Exp_{\gamma,g}(V)(t) = Exp_{\theta*\gamma,g}(\theta*V)(t).    
\end{align*}
We shall find solutions to \eqref{eq:1} in the class of oriented Alexandrov
embedded curves. Let $\gamma \in H^{2,2}(S^1,M)$ be an oriented Alexandrov
embedded curves with corresponding oriented immersion $F$ from $B$ to $M$.
If we equip $B$ with the metric $F^*g$ induced by $F$, then the outer normal $N_B(x)$
at $x \in \rand B$ with respect to $F^*g$ satisfies
\begin{align*}
DF\eval_x N_B(x) = N_\gamma(x)
\end{align*}
where $N_\gamma(x)$ denotes the normal to the curve $\gamma$ at $x \in \rand B$
defined by
\begin{align*}
N_\gamma(x):= |\dot\gamma(x)|^{-1}J_g(\gamma(x))\dot\gamma(x).  
\end{align*}
In \cite{arXiv:0903.1128} the following two basic properties
of oriented Alexandrov embedded curves are shown.
\begin{lemma}
\label{lem:basis_alexandrov}
$ $
\begin{enumerate}
\item Let $(\gamma_n)$ in $C^2(\rand B,M)$ be a sequence of immersions,
which are oriented Alexandrov embedded, such that $(\gamma_n)$ converges 
to an immersion $\gamma_0$ in $C^2(\rand B, M)$ with strictly positive
geodesic curvature.
Then $\gamma_0$ is oriented Alexandrov embedded.
\item The set of regular, oriented Alexandrov embedded curves is open in
$H^{2,2}(S^1,M)$.   
\end{enumerate}
\end{lemma}
Property $(1)$ and $(2)$ are given in \cite{arXiv:0903.1128} for closed
curves in $S^2$. Since the analysis in the 
proof of $(1)$ and $(2)$ is done in tubular neighborhoods
of closed curves, properties $(1)$ and $(2)$ continue to hold
if $S^2$ is replaced by a general surface $M$.  

\section{The $S^{1}$-Poincar\'{e}-Hopf index}
\label{sec:s1_euler}
In \cite{arXiv:0808.4038} a $S^1$-equivariant Poincar\'{e}-Hopf index
or $S^{1}$-degree is introduced for equivariant vector fields on subsets
of $H^{2,2}(S^{1},S^{2})$. The $S^{1}$-degree is based on an
equivariant version of the Sard-Smale lemma \cite[Lem 3.9]{arXiv:0808.4038},
which depends on an appropriate change of a vector field locally around its critical orbits.
It's merely a matter of form to extend this local argument, when $S^{2}$ is replaced
by a general surface $M$. We give a short account of the definition
and properties of the $S^{1}$-degree for equivariant vector fields on subsets
of $H^{2,2}(S^{1},M)$.\\
We define a $C^2$ equivariant vector field $W_g$ on $H^{2,2}(S^1,M)$ by
\begin{align*}
W_{g}(\gamma) = (-(D_{t,g})^2 + 1)^{-1} \dot \gamma, \text{ for } \gamma \in H^{2,2}(S^1,M).  
\end{align*}
We will compute the $S^1$-Poincar\'{e}-Hopf index for the following class of vector fields.
\begin{definition}
Let $\mathscr{M}$ be an open 
$S^1$-invariant subset of prime curves in $H^{2,2}(S^{1},M)$.
A $C^2$ vector field $X$ on $\mathscr{M}$ is called $(\mathscr{M},g,S^{1})$-admissible, if
\begin{itemize}
\item[(1)] $X$ is $S^1$-equivariant, i.e. $X(\theta*\gamma)=\theta*X(\gamma)$ for all
$(\theta,\gamma)\in S^1\times \mathscr{M}$.
\item[(2)] $X$ is proper in $\mathscr{M}$, i.e. the set 
$\{\gamma \in \mathscr{M} \where X(\gamma)=0 \}$ is compact,
\item[(3)] $X$ is orthogonal to $W_{g}$, i.e. 
$\langle X(\gamma), W_g(\gamma)\rangle_{T_\gamma H^{2,2}(S^1,M)}=0$ 
for all $\gamma \in \mathscr{M}$.
\item[(4)] $X$ is a Rothe field, i.e. 
if $X(S^{1}*\gamma)=0$ then $D_{g}X\eval_\gamma$ and 
$\text{Proj}_{\langle W_g(\gamma)\rangle^\perp} \circ D_{g}X\eval_\gamma$
are Rothe maps in $\mathscr{L}(T_\gamma H^{2,2}(S^1,M))$ and
$\mathscr{L}(\langle W_g(\gamma)\rangle^\perp)$, respectively.
\item[(5)] $X$ is elliptic, i.e. there is $\eps>0$ such that  
for all finite sets of charts 
\begin{align*}
\{(Exp_{\gamma_i,g},B_{2\delta_i}(0))\where \gamma_i\in H^{4,2}(S^1,M)
\text{ for }1\le i\le n\},
\end{align*}
and finite sets
\begin{align*}
\{W_i \in T_{\gamma_i}H^{4,2}(S^1,M)\where \|W_i\|_{T_{\gamma_i}H^{4,2}(S^1,M)}<\eps 
\text{ for }1\le i\le n\}, 
\end{align*}
there holds: If
$\alpha \in \capl_{i=1}^n Exp_{\gamma_i,g}(B_{\delta_i}(0))\subset H^{2,2}(S^1,M)$
satisfies
\begin{align*}
X(\alpha) = \sum_{i=1}^n \text{Proj}_{\langle W_g(\alpha)\rangle^\perp}\circ
D Exp_{\gamma_i,g}\eval_{Exp_{\gamma_i,g}^{-1}(\alpha)}(W_i)  
\end{align*}
then $\alpha$ is in $H^{4,2}(S^1,M)$.
\end{itemize}
\end{definition}
It is shown in \cite{arXiv:0808.4038} that $X_{k,g}$ satisfies properties
$(3)-(4)$. Hence, $X_{k,g}$ is $(\mathscr{M},g,S^{1})$-admissible if and only
if $X_{k,g}$ is proper in $\mathscr{M}$.
Note that the regularity property $(5)$, taking $W_i=0$, shows that
any zero of $X$ belongs to $H^{4,2}(S^1,M)$. 
Furthermore, for $\gamma \in H^{4,2}(S^1,M)$ the map
$\theta \mapsto \theta*\gamma$ is $C^2$ from $S^1$ to $H^{2,2}(S^{1},M)$. Hence, if
$X(\gamma)=0$ then
\begin{align*}
0 &= D_\theta(X(\theta*\gamma))\eval_{\theta=0} = D_{g}X\eval_\gamma (\dot\gamma),
\end{align*}
such that the kernel of $D_{g}X\eval_\gamma$ is nontrivial.
If $X$ is a vector field orthogonal to $W_{g}$ 
and $X(\gamma)=0$, then
\begin{align*}
0 &= D\big(\langle X(\alpha),W_g(\alpha)\rangle_{T_\alpha H^{2,2}(S^{1},M),g}\big)\eval_\gamma
= \langle D_{g}X\eval_\gamma,W_g(\gamma)\rangle_{T_\gamma H^{2,2}(S^{1},M),g}
\end{align*}
where the various curvature terms and terms containing derivatives 
of $W_g$ vanish as $X(\gamma)=0$.  
Thus, $X(\gamma)=0$ implies
\begin{align}
\label{eq:dx_to_wg_perp}
D_{g}X\eval_\gamma:\:T_\gamma H^{2,2}(S^{1},M) \to \langle W_g(\gamma)\rangle^{\perp},
\end{align}
and the projection $\text{Proj}_{\langle W_g(\gamma)\rangle^{\perp}}$ in $(4)$ is unnecessary.
\begin{definition}
\label{d:critical_orbit}
Let $\mathscr{M}$ be an open 
$S^1$-invariant subset of prime curves in $H^{2,2}(S^{1},M)$, $S^{1}*\gamma \subset \mathscr{M}$,
and $X$ a $(\mathscr{M},g,S^{1})$-admissible vector field on $\mathscr{M}$.\\  
The orbit $S^{1}*\gamma$ is called a {\em critical
orbit} of $X$, if $X(\gamma)=0$.\\ 
The orbit $S^{1}*\gamma$ is called a {\em nondegenerate critical
orbit} of $X$, if $X(\gamma)=0$ and 
\begin{align*}
D_{g}X\eval_\gamma:\: \langle W_{g}(\gamma)\rangle^{\perp} 
\xpfeil{}{} \langle W_{g}(\gamma)\rangle^{\perp}   
\end{align*}
is an isomorphism.
\end{definition}
Note that if $\gamma \in H^{4,2}(S^1,M)\subset H^{2,2}(S^1,M) $ then 
$\dot\gamma \not\in \langle W_{g}(\gamma)\rangle^{\perp}$. 
\begin{definition}
\label{d:M_g_t_s_1_homotopy}  
Let $\{g_t \where t \in [0,1]\}$ be a family of smooth metrics on $M$, which induces
a corresponding family of metrics on $H^{2,2}(S^1,M)$, still denoted by $g_t$. 
Let $\mathscr{M}$ be an open 
$S^1$-invariant subset of prime curves in $H^{2,2}(S^{1},M)$ and $X_0$, $X_1$ 
two vector-fields on $\mathscr{M}$ such that
$X_i$ is $(\mathscr{M},g_i,S^1)$-admissible for $i=0,1$.
A $C^2$ family of vector-fields $X(t,\cdot)$ on $\mathscr{M}$ for $t\in [0,1]$ is called 
a $(\mathscr{M},g_t,S^1)$-homotopy between $X_0$ and $X_1$, if
\begin{itemize}
\item $X(0,\cdot)=X_0$ and $X(1,\cdot)=X_1$,
\item $\{(t,\gamma)\in[0,1]\times \mathscr{M} \where X(t,\gamma)=0\}$ is compact,
\item $X_t:= X(t,\cdot)$ is $(\mathscr{M},g_t,S^1)$-admissible for all $t\in [0,1]$.
\end{itemize}
We write $(\mathscr{M},g,S^{1})$-homotopy, 
if the family of metrics $\{g_t\}$ is constant.
\end{definition}
Note that,
if $\{k_t \in C^\infty(M,\rz)\where t \in [0,1]\}$ is a $C^2$ family of smooth function,
then $t \mapsto X_{k_t,g_t}$ is a $(\mathscr{M},g_t,S^1)$-homotopy, if and only if
the set 
\begin{align*}
\{(t,\gamma)\in[0,1]\times \mathscr{M} \where X_{k_t,g_t}(\gamma)=0\}
\end{align*}
is compact.\\
We let $\mathscr{M}$ be an open 
$S^1$-invariant subset of prime curves in $H^{2,2}(S^{1},M)$ and $X$
a $(\mathscr{M},g,S^{1})$-admissible vector field on $\mathscr{M}$.
The local $S^1$-degree of an isolated, nondegenerate critical orbit $S^1*\gamma_0$
is defined by
\begin{align*}
\deg_{loc,S^1}(X,S^1*\gamma_0):=\sgn D_g X\eval_{\gamma_0},
\end{align*}
where $\sgn D_g X\eval_{\gamma_0}$ is the sign of the Rothe map $D_g X\eval_{\gamma_0}$
in $\mathscr{L}(\langle W_{g}(\gamma)\rangle^{\perp})$. Since $D_g X_{k,g}\eval_{\gamma_0}$
is of the form $identity+compact$, in the above situation $\sgn D_g X_{k,g}\eval_{\gamma_0}$
is given by the usual Leray-Schauder degree.\\
Using an equivariant version of the Sard-Smale lemma a $S^1$-equivariant Poincar\'{e}-Hopf index
\begin{align*}
\chi(X,\mathscr{M}) \in \zz  
\end{align*}
is defined in \cite{arXiv:0808.4038} with the following properties.
\begin{lemma}
\label{lem:degree_properties}
$ $
\begin{enumerate}
\item If $X$ is $(\mathscr{M},g,S^{1})$-admissible with only finitely many critical orbits,
that are all nondegenerate, then
\begin{align*}
\chi_{S^1}(X,\mathscr{M}) := \sum_{\{S^{1}*\gamma\subset \mathscr{M} \where X(S^{1}*\gamma)=0\}}
\deg_{loc,S^1}(X,S^{1}*\gamma).  
\end{align*}
\item If $X_0$ and $X_1$ are $(\mathscr{M},g_t,S^1)$-homotop, then
$\chi(X_0,\mathscr{M})=\chi(X_1,\mathscr{M})$.
\end{enumerate}  
\end{lemma}

\section{The Unperturbed Problem}
\label{sec:unperturbed}
Let $\hz \subset \rz^{3}$ be the standard hyperbolic plane
\begin{align*}
\hz:= \{(\xi^1,\xi^2,\tau) \in \rz^3 \where \tau^2-|\xi|^2=1 \text{ and }\tau >0\}  
\end{align*}
with metric $g_{0}$ induced by the Minkowski metric $g_m$,
\begin{align*}
g_m := (d\xi^1)^2+(d\xi^2)^2-(d\tau)^2=\langle \cdot,\cdot\rangle_m.
\end{align*}
We choose the orientation on $\hz$ 
such that $J_{g_0}(y)$ is given for $y \in \hz$ by
\begin{align*}
J_{g_0}(y)(v) := y \cross_m v \text{ for all }v \in T_y \hz,   
\end{align*}
where $\cross_m$ denotes the twisted cross product in $\rz^{3}$,
\begin{align*}
\begin{pmatrix}
v^1\\v^2\\v^3    
\end{pmatrix}
\cross_m
\begin{pmatrix}
w^1\\w^2\\w^3    
\end{pmatrix}
:=
\begin{pmatrix}
v^3w^2-v^2w^3\\v^1w^3-v^3w^1\\v^1w^2-v^2w^1    
\end{pmatrix}.
\end{align*}
The twisted cross product $\cross_m$ is related 
to the usual cross product $\cross$ in $\rz^3$ by  
$v \cross_m w = I_{2,1}v \cross I_{2,1}w$, where 
$I_{2,1}$ is given by
\begin{align*}
I_{2,1}:= 
\begin{pmatrix}
1 &0 &0\\
0 &1 &0\\
0&0&-1    
\end{pmatrix},
\end{align*}
and satisfies for $a,\,b,\,c,\,d \in \rz^3$ 
\begin{align*}
\langle (a\cross_m b),a\rangle_m=0=\langle (a\cross_m b),b\rangle_m,\\
a\cross_m(b\cross_m c)= -b\langle a,c\rangle_m + c\langle a,b\rangle_m,\\
\langle (a\cross_m b),(c \cross_m d)\rangle_m
= -\langle a,c\rangle_m \langle b,d\rangle_m +
\langle b,c\rangle_m\langle a,d\rangle_m   
\end{align*}
We fix a compact, orientable Riemannian surface $(M,g_0)$,
\begin{align*}
M := \hz/\Gamma,  
\end{align*}
where $\Gamma \subset SO(2,1)_+$ is a group
of oriented isometries acting
freely and properly discontinuously on $\hz$. Concerning the metric we will
be sloppy and denote by $g_0$ the
metric on $\hz$ as well as the induced metric on $\hz/\Gamma$.
The unperturbed problem on $M$ is given by
\begin{align}
\label{eq:unperturbed_M}
D_{t,g_0} \dot \gamma = |\dot \gamma|_{g_0} k_0 J_{g_0}(\gamma) \dot \gamma,  
\end{align}
where $k_0$ is a positive constant.\\
We shall compute the $S^1$-degree of the unperturbed equation \eqref{eq:unperturbed_M}
in three steps. {\em Step 1:} We compute explicitly the set $\mathcal{Z}_M$
of Alexandrov embedded solutions in $H^{2,2}(S^1,M)$ to \eqref{eq:unperturbed_M}
and show that $\mathcal{Z}_M$ is a finite dimensional, nondegenerate manifold, 
in the sense that we have for all $\tilde{\alpha} \in \mathcal{Z}_M$
\begin{align*}
T_{\tilde{\alpha}} \mathcal{Z}_M = \text{kernel}(D_{g_0}X_{k_0,g_0}\eval_{\tilde{\alpha}}),\\
T_{\tilde{\alpha}} H^{2,2}(S^1,M) = 
T_{\tilde{\alpha}} \mathcal{Z}_M \oplus R(D_{g_0}X_{k_0,g_0}\eval_{\tilde{\alpha}}).  
\end{align*}
{\em Step 2:} In Section \ref{sec:perturbed} 
we perform a finite dimensional reduction of a slightly perturbed problem:
We consider for $k_1\in C^\infty(M,\rz)$, which will be chosen later,
and $\eps \in \rz$, which is assumed to be very small,
the perturbed vector field $X_{g_0,\eps}$ defined by
\begin{align*}
X_{g_0,\eps}(\gamma)&:=
(-D_{t,g_0}^{2} + 1)^{-1} 
\big(-D_{t,g_0} \dot\gamma + |\dot \gamma|_{g_0} 
(k_0+\eps k_1(\gamma))J_{g_0}(\gamma)\dot\gamma\big)\\
&= X_{k_0,g_0}(\gamma)+\eps K_1(\gamma),
\end{align*}
where the vector field $K_1$ is given by
\begin{align*}
K_1(\gamma):= (-D_{t,g_0}^{2} + 1)^{-1} |\dot \gamma|_{g_0} 
\big(k_1(\gamma)J_{g_0}(\gamma)\dot\gamma \big).  
\end{align*}
We show that if $S^1*\tilde{\alpha}_0 \subset \mathcal{Z}_M$ is a nondegenerate
critical orbit 
of the vector field $\tilde{\alpha} \mapsto P_1(\tilde{\alpha}) \circ K_1(\tilde{\alpha})$
on $\mathcal{Z}_M$, 
where $P_1(\tilde{\alpha})$ is a projection onto $T_{\tilde{\alpha}} \mathcal{Z}_M$ defined below,
then for any $0<\eps <<1$ 
there is a unique nondegenerate critical orbit $S^1*\tilde{\gamma}(\eps)$
of $X_{g_0,\eps}$ 
 such that $\tilde{\gamma}(\eps)$ converges to $\tilde{\alpha}_0$ as $\eps \to 0^+$
and
\begin{align*}
\deg_{loc, S^{1}}(X_{g_0,\eps},S^1*\gamma(\eps)) 
&= -\deg_{loc}(P_1(\cdot)\circ K_1(\cdot),S^1*\tilde{\alpha}_0).  
\end{align*}  
{\em Step 3:} In Section \ref{sec:computation_degree}
we choose a Morse function $k_1 \in C^\infty(M,\rz)$ with critical points
\begin{align*}
\{\tilde{w}_i \in M \where 1 \le i \le n\}.
\end{align*}
We show that if $k_0 >>1 $ is large, then $P_1(\cdot)\circ K_1(\cdot)$ has exactly 
$n$ critical orbits $\{S^1*\tilde{\alpha}_{i,k_0} \where 1 \le i \le n\}$ such that
for $1 \le i \le n$
\begin{align*}
\deg_{loc}(P_1(\cdot)\circ K_1(\cdot),S^1*\tilde{\alpha}_{i,k_0})= 
\deg_{loc}(\nabla k_1, \tilde{w}_i). 
\end{align*}
This yields the formula
$\chi_{S^1}(X_{k_0,g_0},\mathscr{M}_A)=-\chi(M)$, where $\mathscr{M}_A$ is the subset of 
$H^{2,2}(S^{1},M)$ consisting of Alexandrov embedded, regular curves.\\
{\em Step 1:}
The prescribed geodesic curvature equation with $k\equiv k_0$ on $(\hz,g_0)$ is given by
\begin{align}
\label{eq:unpert:1}
Proj_{\gamma^\perp,g_m}\ddot \gamma = |\dot \gamma|_m k_0 \gamma \cross_m \dot \gamma,  
\end{align}
where $\gamma \in H^{2,2}(S^1,\hz)$, $\dot \gamma$ and $\ddot \gamma$ are the usual derivatives
of $\gamma$ considered as a curve in $\rz^3$, $|\dot \gamma|_m$ is the Minkowski norm 
of $\dot \gamma$
in $(\rz^3,g_m)$.\\
If $k_0>1$ then there is a unique $r=r(k_0)>0$ such that
\begin{align*}
k_0= \frac{\sqrt{1+r^2}}{r}.  
\end{align*}
We call a triple of vectors $\{v_0,v_1,w\}$ in $\rz^3$ a positive oriented orthonormal system
with respect to $g_m$, if
\begin{align*}
\langle v_0,v_1\rangle_m = \langle v_0,w\rangle_m= \langle v_1,w\rangle_m=0,\\
\langle v_0,v_0\rangle_m =\langle v_1,v_1\rangle_m = -\langle w,w\rangle_m=1,\\
v_0 \cross_m v_1 = w.
\end{align*}
We define for $\lambda>0$ and a positive oriented orthonormal system $\{v_0,v_1,w\}$
the function $\alpha \in C^\infty(\rz,\hz)$ by
\begin{align}
\label{eq:soln_alpha}
\alpha(t,\lambda,v_0,v_1,w) := \sqrt{1+r^2}w+r\cos(\lambda r^{-1}t)v_1+r\sin(\lambda r^{-1}t)v_0  
\end{align}
A direct calculation shows that $\alpha(\cdot,\lambda,v_0,v_1,w)$ solves (\ref{eq:unpert:1}).
We fix $(\gamma_0,\tilde{v}_0) \in T\hz$ with $\tilde{v}_0 \neq 0$
and define
the parameter $\lambda := |\tilde{v}_0|_m$ and 
the positive oriented orthonormal system $(v_0,v_1,w)$ by
\begin{align*}
v_0 &:= \lambda^{-1} \tilde{v}_0,\, v_1 := -r\gamma_0-\sqrt{1+r^2}(v_0\cross_m \gamma_0),\,
w := v_0\cross_m v_1.
\end{align*}
Then $\alpha(\cdot,\lambda,v_0,v_1,w)$ satisfies the initial conditions
\begin{align*}
\alpha(0,\lambda,v_0,v_1,w)=\gamma_0,\; \dot\alpha(0,\lambda,v_0,v_1,w)= \tilde{v}_0,  
\end{align*}
and we deduce that all non constant solutions of \eqref{eq:unpert:1} are obtained in this way.
Since we are only interested in solutions in $H^{2,2}(S^1,\hz)$ we get an extra condition
on $\lambda$, i.e. the $1$-periodicity leads to
\begin{align*}
\lambda \in 2\pi \nz r.  
\end{align*}
\begin{lemma}
\label{l:alexandrov_soln_hz}
The oriented Alexandrov embedded solutions in 
$H^{2,2}(S^1,\hz)$ of equation (\ref{eq:unpert:1}) are
given by the set of simple solutions
\begin{align*}
\mathcal{Z}_{\hz} :=
\big\{&\alpha(\cdot, 2\pi r,v_0,v_1,w )\where \\
&\{v_0,v_1,w\} 
\text{ is a pos. orth. system in }(\rz^3,g_m)\big\}.  
\end{align*} 
\end{lemma}
\begin{proof}
From the analysis above the periodic solutions to (\ref{eq:unpert:1}) are given by
\begin{align*}
\big\{\alpha(\cdot, 2\pi n r,v_0,v_1,w )\where 
&n \in \nz \text{ and }\\
&\{v_0,v_1,w\} 
\text{ is a pos. orth. system in }(\rz^3,g_m)\big\}.  
\end{align*} 
We fix $n \in \nz$ and a positive orthonormal system $\{v_0,v_1,w\}$ and write
\begin{align*}
\gamma_n:= \alpha(\cdot, 2\pi n r,v_0,v_1,w).  
\end{align*}
Assume $\gamma_n$ is oriented Alexandrov embedded and let $F_n$ be the 
corresponding immersion.
Since $\hz$ is diffeomorphic to $\rz^2$
we may assume 
that $\gamma_1$ is a simple curve in the plane $(\rz^2,\delta)$ with standard metric $\delta$.
If we apply the Gau{\ss}-Bonnet formula to $(B,F_n^*\delta)$ and the embedded
curve $\gamma_1$ in the plane, we obtain
\begin{align*}
2\pi &= \int_{\rand B} k_{F_n^*\delta}\, dS_{F_n^*\delta} + \int_B K_{F_n^*\delta}\, dA_{F_n^*\delta}\\ 
&= \int_{\gamma_n} k_{\delta} \, dS_{\delta}
= n\int_{\gamma_1} k_{\delta} \, dS_{\delta} = n2\pi,
\end{align*}
which is only possible for $n=1$.\\
The curve $\gamma_1$ is oriented Alexandrov embedded using polar coordinates and
\begin{align*}
[0,2\pi]\times [0,1] \ni (t,s) \mapsto 
\sqrt{1+s^2r^2}w+sr\cos(t)v_1+sr\sin(t)v_0.
\end{align*}
\end{proof}
The Lorentz transformations 
$S0(2,1)_+$ of $(\rz^3,g_m)$,
\begin{align*}
SO(2,1)_+ := \{A \in O(2,1) \where A(\hz)\subset \hz \text{ and } \det A = 1\},  
\end{align*}
correspond to the oriented isometries of $(\hz,g_0)$ and act on solutions: 
if $\gamma$ solves (\ref{eq:unpert:1}) so does $A\circ\gamma$
for any $A\in SO(2,1)_+$. We have
\begin{align*}
A\circ \alpha(\cdot, \lambda,v_0,v_1,w ) = \alpha(\cdot, \lambda,A(v_0),A(v_1),A(w)).
\end{align*}
Moreover, there holds,
\begin{align}
\label{eq:param_crit_orbits_hz}
\alpha(\cdot, 2\pi r,v_0,v_1,w ) = \theta*\alpha(\cdot, 2\pi r,v_0',v_1',w')
\end{align}
for some $\theta \in S^1$ if and only if $w=w'$. 
Consequently, the critical orbits of \eqref{eq:unpert:1} in $\hz$,
$
\{S^{1}*\gamma \where \gamma \in \mathcal{Z}_{\hz}\},  
$
are parametrized by $w \in \hz$ and correspond to ``circles'' with radius $r$
around the center $w$ in $\hz$.\\
We let $\pi_M$ be the natural projection, $\pi_M: \hz \to \hz/\Gamma$.
Any point $z \in \hz$ admits a neighborhood $U=B_{\delta}(z)$ 
such that $\pi_M\eval_U: U \to  \pi_M(U)$
is an isometry. From \eqref{eq:soln_alpha} there is $C_{k_0}>1$ such that
if $k_0\ge C_{k_0}$ then any solution to \eqref{eq:unpert:1} on $\hz$ passing through $z$
remains in $U$. For $M$ is compact $C_{k_0}= C_{k_0}(\Gamma)$ and $\delta>0$ may be chosen 
independently of $z$. Equation \eqref{eq:1} is invariant under isometries,
hence the set of solutions to \eqref{eq:unperturbed_M} with $k_0 \ge C_{k_0}$
is given by
\begin{align*}
\big\{&\pi_M\circ \alpha(\cdot, 2\pi r,v_0,v_1,w )\where \\
&\{v_0,v_1,w\} 
\text{ is a pos. orth. system in }(\rz^3,g_m)\big\}.   
\end{align*}
Moreover, we have
\begin{lemma}
\label{l:alexandrov_soln_M}
If $k_0 \ge C_{k_0}$, then
the oriented Alexandrov embedded solutions in 
$H^{2,2}(S^1,M)$ of equation \eqref{eq:unperturbed_M} are
given by the set of simple solutions
\begin{align*}
\mathcal{Z}_{M} :=
\big\{&\tilde{\alpha}=\pi_M \circ \alpha(\cdot, 2\pi r,v_0,v_1,w )\where \\
&\{v_0,v_1,w\} 
\text{ is a pos. orth. system in }(\rz^3,g_m)\big\}.  
\end{align*}  
\end{lemma}
\begin{proof}
We fix $n \in \nz$ and a positive orthonormal system $\{v_0,v_1,w\}$ and write
\begin{align*}
\gamma_n:= \pi_M \circ \alpha(\cdot, 2\pi n r,v_0,v_1,w)  
\end{align*}
From the above analysis any periodic solution to \eqref{eq:unperturbed_M} on $(M,g_0)$ is
of this form. 
Hence, it is enough 
to show that 
$\gamma_n$ is oriented Alexandrov embedded, if and only if $n=1$.\\
Concatenating the immersion in the proof of Lemma \ref{l:alexandrov_soln_hz}
with $\pi_M$ we deduce that $\gamma_1$ is oriented Alexandrov embedded.
Suppose $\gamma_n$ is oriented Alexandrov embedded with an immersion $F_n:B \to M$.
From the homotopy lifting property of the covering $\pi_M:\hz \to M$ we may lift $F_n$
to see that $\alpha(\cdot, 2\pi n r,v_0,v_1,w)$ is oriented Alexandrov embedded
in $\hz$. From Lemma \ref{l:alexandrov_soln_hz} this is only possible for $n=1$. 
\end{proof}
From \eqref{eq:param_crit_orbits_hz} we find
\begin{align*}
\pi_M\circ \alpha(\cdot, 2\pi r,v_0,v_1,w ) = \theta* \pi_M\circ \alpha(\cdot, 2\pi r,v_0',v_1',w')
\end{align*}
for some $\theta \in S^1$ if and only if $\pi_M(w)=\pi_M(w')$, 
such that the critical orbits of \eqref{eq:unpert:1} in $M$
are parametrized by $w \in M$
and correspond to projections on $M$ of ``circles'' in $\hz$.\\
In the following we always assume that 
\begin{align*}
k_0 \ge C_{k_0}.  
\end{align*}
We denote by $X_{k_0,g_0,\hz}$ the vector field on $H^{2,2}(S^1,\hz)$ corresponding
to equation \eqref{eq:unpert:1}. 
We fix a solution $\alpha=\alpha(\cdot,2\pi r,v_0,v_1,w)$ of \eqref{eq:unpert:1}  
and note that for $V \in  T_{\alpha}H^{2,2}(S^1,\hz)$
\begin{align*}
R_{g_0}(V,\dot\alpha)\dot\alpha = -V|\dot \alpha|_m^{2}+\langle V,\dot\alpha\rangle_m \dot\alpha.  
\end{align*}
By \eqref{eq:dg_x_g_formula} a vector field $W$ is contained in
the kernel of $D_{g_0}X_{k_0,g_0,\hz}\eval_{\alpha}$ if and only if $W$ is a periodic solution
of
\begin{align}
\label{eq:linearization}
0 &= 
-D_{t,g_0}^2W + W|\dot \alpha|_m^{2}-\langle W,\dot\alpha\rangle_m \dot\alpha \notag \\
&\quad + |\dot\alpha|_m^{-1}\langle D_{t,g_0}W,\dot\alpha\rangle_m k_0 (\alpha\cross_m \dot\alpha)
+|\dot\alpha|_m  k_0 (\alpha \cross_m D_{t,g_0}W).  
\end{align}
Due to the geometric origin of equation (\ref{eq:unpert:1}) and the
$SO(2,1)_+$ invariance we find that
\begin{align}
\label{eq:12}
W_0(t,v_0,v_1,w) &:= t\dot\alpha,\\
W_1(t,v_0,v_1,w) &:=  \dot\alpha= 2\pi r(-\sin(2\pi t)v_1 + \cos(2\pi t)v_0), \notag \\
W_2(t,v_0,v_1,w) &:=  (1+r^2)^{\frac12} v_1+r\cos(2\pi t)w, \notag \\
W_3(t,v_0,v_1,w) &:= (1+r^2)^{\frac12} v_0+r\sin(2\pi t)w,\notag
\end{align}
solve \eqref{eq:linearization}. In the sequel,
we will omit the dependence of $W_i$ on $(v_0,v_1,w)$, if there is no possibility of confusion.
The initial values of $W_0,\dots, W_3$ 
\begin{align*}
W_0(0,v_0,v_1,w)& =0 ,\, 
D_{t,g_0} W_0(0,v_0,v_1,w)= 2\pi r v_0,\\
W_1(0,v_0,v_1,w) &=2\pi r v_0,\, D_{t,g_0} W_1(0,v_0,v_1,w)= 
-4\pi^2 r^3k_0(k_0v_1+w),\\
W_2(0,v_0,v_1,w) &= rk_0 v_1+rw,\, D_{t,g_0} W_2(0,v_0,v_1,w)= 0,\notag\\  
W_3(0,v_0,v_1,w) &= r k_0 v_0,\, D_{t,g_0} W_3(0,v_0,v_1,w)= -2\pi r^{3}(k_0v_1+w).
\end{align*}
are a basis of $\big(T_{\alpha(0)}\hz\big)^2$, such that any solution to \eqref{eq:linearization}
is a linear combination of $W_0,\dots, W_3$.
As only $W_1,\dots, W_3$ are periodic, we obtain
\begin{align}
\label{eq:kernel_DX_0}
\text{kernel}(D_{g_0}X_{k_0,g_0,\hz}\eval_{\alpha})=
\langle W_1,\,W_2,\,W_3 \rangle.  
\end{align}
We fix a neighborhood $U$ of $\alpha(0)$ as above, where $\pi_M:\: U \to \pi_M(U)$
is an isometry. Then $\alpha \in H^{2,2}(S^1,U)$ and $\pi_M$ induces isomorphisms
\begin{align*}
\pi_M:\: H^{2,2}(S^1,U) \to H^{2,2}(S^1,\pi_M(U)), \, \alpha \mapsto \pi_M \circ \alpha,\\
(\pi_M)_*:\: T_\alpha H^{2,2}(S^1,\hz) \to T_{\pi_M \circ \alpha} H^{2,2}(S^1,M), 
\, V \mapsto d\pi_M\eval_\alpha V,   
\end{align*}
where $(\pi_M)_*$ is an isometry. Moreover, there holds on $H^{2,2}(S^1,U)$ 
\begin{align}
\label{eq:3}
(\pi_M)_* \circ X_{g_0,k_0,\hz} = X_{g_0,k_0}\circ \pi_M, \notag \\
(\pi_M)_* \circ D_{g_0}X_{g_0,k_0,\hz}\eval_\alpha 
=D_{g_0} X_{g_0,k_0}\eval_{\pi_M \circ \alpha} \circ (\pi_M)_*. 
\end{align}
Since
$\mathcal{Z}_{\hz}$ 
and $\mathcal{Z}_{M}$ are three dimensional submanifolds of $H^{2,2}(S^{1},\hz)$
and $H^{2,2}(S^{1},M)$, respectively, we have for $\alpha \in \mathcal{Z}_{\hz}$
and $\tilde{\alpha}=\pi_M \circ \alpha \in \mathcal{Z}_{M}$ 
\begin{align*}
T_\alpha \mathcal{Z}_{\hz} &= 
\text{kernel}(D_{g_0}X_{k_0,g_0,\hz}\eval_{\alpha})=\langle W_1,\,W_2,\,W_3 \rangle,\\
T_{\tilde{\alpha}} \mathcal{Z}_{M} &= 
\text{kernel}(D_{g_0}X_{k_0,g_0}\eval_{\tilde{\alpha}})\\
&=\langle \tilde{W}_i:= (\pi_M)_* \circ W_i \where 1 \le i \le 3 \rangle.
\end{align*} 
To compute the image of $D_{g_0}X_{k_0,g_0,\hz}\eval_{\alpha}$ we note 
that
$\{\dot\alpha,\alpha\cross_m\dot\alpha\}$
is an orthogonal system in $T_{\alpha}\hz$ for any $t\in S^{1}$. 
Thus any $V \in T_\alpha H^{2,2}(S^1,\hz)$ may be written
as
\begin{align*}
V= \lambda_1 \dot\alpha + \lambda_2 (\alpha\cross_m \dot\alpha)   
\end{align*}
for some functions $\lambda_1,\,\lambda_2 \in H^{2,2}(S^{1},\rz)$.
Using the fact that
\begin{align*}
D_{t,g_0} \dot\alpha &= |\dot \alpha|_m k_0(\alpha \cross_m \dot\alpha) \text{ and }
D_{t,g_0} (\alpha \cross_m \dot\alpha) = -|\dot \alpha|_m k_0 \dot\alpha,    
\end{align*}
we obtain
\begin{align}
\label{eq:d_x_g0_alpha}
D_{g_0}X_{k_0,g_0,\hz}\eval_{\alpha}(V) &= (-D_{t,g_0}^2+1)^{-1}
\big((-\lambda_1''+2\pi\sqrt{1+r^{2}} \lambda_2') \dot\alpha \notag\\
&\qquad +(-\lambda_2''-(2\pi)^{2}\lambda_2)(\alpha\cross_m \dot\alpha)\big). 
\end{align}
Concerning $W_1,\dots, W_3$ and $W_{g_0}$ we find 
\begin{align}
W_1(t) &= \dot\alpha(t),\notag \\
W_2(t) &= -\frac{1}{2\pi r}\big(\sqrt{1+r^{2}} \sin(2\pi t)\dot\alpha(t) + \cos(2\pi t)(\alpha\cross_m
\dot\alpha)\big),\notag \\
W_3(t) &= -\frac{1}{2\pi r}\big(-\sqrt{1+r^{2}} \cos(2\pi t)\dot\alpha(t) 
+ \sin(2\pi t)(\alpha\cross_m \dot\alpha)\big)\notag \\
\label{eq:8}
W_{g_0}(\alpha) &= (1+|\dot\alpha|_m^{2}k_0^{2})^{-1}\dot\alpha 
=  (1+4\pi^2(1+r^2))^{-1} W_1. 
\end{align}

\begin{lemma}
\label{l:kernel_not_in_range}
If $r\neq (2\pi)^{-1}$, then we have
for $\alpha \in \mathcal{Z}_{\hz}$
\begin{align*}
\{0\}&=\langle W_1,W_2,W_3\rangle \cap R\big(D_{g_0}X_{k_0,g_0,\hz}\eval_{\alpha}\big) ,\\
\langle W_1\rangle^{\perp}&=  
\langle W_2,W_3\rangle \oplus R\big(D_{g_0}X_{k_0,g_0,\hz}\eval_{\alpha}\big)
\end{align*}
\end{lemma}
\begin{proof}
For $\lambda_1,\lambda_2 \in H^{2,2}(S^{1},\rz)$ we have
\begin{align*}
 (-D_{t,g_0}^2+1)&\big(\lambda_1 \dot\alpha + \lambda_2(\alpha \cross_m \dot\alpha)\big)\\
&= \big(-\lambda_1''+4\pi\sqrt{1+r^{2}} \lambda_2' +(4\pi^{2}(1+r^{2})+1) \lambda_1\big)\dot\alpha\\
&\quad +\big(-\lambda_2''-4\pi\sqrt{1+r^{2}} \lambda_1' +(4\pi^{2}(1+r^{2})+1) \lambda_2\big)
\alpha \cross_m \dot\alpha
\end{align*}
Hence we get by direct calculations
\begin{align}
(-D_{t,g_0}^2+1)(W_1) &= (4\pi^{2}(1+r^{2})+1) \dot\alpha,\notag \\
(-D_{t,g_0}^2+1)(-2\pi r W_2) &= \sqrt{1+r^{2}}(4\pi^{2}r^{2} +1) \sin(2\pi t) \dot\alpha \notag \\
\label{eq:D2_1_w2}
&\quad + (-4\pi^{2} r^{2}+1)\cos(2\pi t) (\alpha \cross_m \dot\alpha),\\
(-D_{t,g_0}^2+1)(-2\pi r W_3) &= -\sqrt{1+r^{2}}(4\pi^{2}r^{2} +1) \cos(2\pi t) \dot\alpha\notag\\
\label{eq:D2_1_w3}
&\quad + (-4\pi^{2} r^{2}+1)\sin(2\pi t) (\alpha \cross_m \dot\alpha).
\end{align}
Consequently, by (\ref{eq:dx_to_wg_perp}), \eqref{eq:8}, and the above computations
$W_1$ is orthogonal  
to $\langle W_2,W_3\rangle$ and to $R\big(D_{g_0}X_{k_0,g_0,\hz}\eval_{\alpha}\big)$
in $T_\alpha H^{2,2}(S^1,\hz)$.
As in $L^{2}(S^{1},\rz)$
\begin{align*}
\lambda_2''+(2\pi)^{2}\lambda_2 \perp_{L^2} \langle \cos(2\pi t),\sin(2\pi t)\rangle,\;
\langle \lambda_1'',\lambda_2'\rangle  \perp_{L^2} {\rm const},
\end{align*}
we get from \eqref{eq:d_x_g0_alpha} and the fact that $1-4\pi^2r^2\neq 0$
\begin{align*}
\{0\} &=  (-D_{t,g_0}^2+1)\big(\langle W_1,W_2,W_3\rangle\big)\\
&\qquad \cap 
(-D_{t,g_0}^2+1)D_{g_0}X_{k_0,g_0,\hz}\eval_{\alpha}(T_{\alpha}H^{2,2}(S^{1},\hz))   
\end{align*}
and the claim follows for $D_{g_0}X_{k_0,g_0,\hz}\eval_{\alpha}$ is a Fredholm operator of index $0$.
\end{proof}
Moreover, we see for $\alpha \in \mathcal{Z}_{\hz}$
\begin{align}
\label{eq:4}
R(D_{g_0}X_{k_0,g_0,\hz}\eval_{\alpha}) &= 
\big\{
(-D_{t,g_0}^2+1)^{-1}\big(
(-\lambda_1''+2\pi\sqrt{1+r^2}\lambda_2')\dot{\alpha} \notag\\
&\qquad -(\lambda_2''+(2\pi)^2\lambda_2)(\alpha \cross_m \dot{\alpha})\big)
\where \lambda_1,\,\lambda_2 \in H^{2,2}(S^1,\rz)
\big\} \notag \\
&=
\big\{
(-D_{t,g_0}^2+1)^{-1}\big(
\lambda_1\dot{\alpha}+\lambda_2(\alpha \cross_m \dot{\alpha})\big)
\where
\lambda_i \in L^2(S^1,\rz), \notag \\
&\qquad  \lambda_1 \perp_{L^2} 1,\,
\lambda_2 \perp_{L^2} \langle \cos(2\pi t),\sin(2\pi t)\rangle
\big\}\notag \\
&= \langle (\alpha \cross_m \dot{\alpha}) \rangle \oplus E_{+},
\end{align}
where $E_{+}$ is given by
\begin{align*}
E_{+} &=
\big\{
(-D_{t,g_0}^2+1)^{-1}\big(
\lambda_1\dot{\alpha}+\lambda_2(\alpha \cross_m \dot{\alpha})\big)
\where\\
&\qquad \lambda_i\in L^2(S^1,\rz),\, \lambda_1 \perp_{L^2}1,\,
\lambda_2 \perp_{L^2} \langle 1, \cos(2\pi t),\sin(2\pi t)\rangle
\big\}   
\end{align*}
We have for $V=\lambda_1\dot{\alpha}+\lambda_2 (\alpha \cross_m \dot{\alpha})$ in $T_\alpha H^{2,2}(S^{1},\hz)$ 
\begin{align*}
D_{g_0}X_{k_0,g_0,\hz}\eval_{\alpha}(V) \in E_{+} \aequi \lambda_2 \perp_{L^{2}} 1
\aequi V \perp_{L^{2}} (\alpha \cross_m \dot{\alpha}).  
\end{align*}
We fix
\begin{align*}
V=(-D_{t,g_0}^{2}+1)^{-1}(\lambda_1\dot{\alpha}+\lambda_2 (\alpha \cross_m \dot{\alpha})) 
\in E_{+}.  
\end{align*} 
Then
\begin{align*}
\int_{S^{1}} \langle V,&\alpha \cross_m \dot{\alpha}\rangle_m\\
&=
\int_{S^{1}} \langle
(-D_{t,g_0}^{2}+1)^{-1}(\lambda_1\dot{\alpha}+\lambda_2 (\alpha \cross_m \dot{\alpha})),
\alpha \cross_m \dot{\alpha}\rangle_m\\
&= \int_{S^{1}} \langle\lambda_1\dot{\alpha}+\lambda_2 (\alpha \cross_m \dot{\alpha}) 
,(-D_{t,g_0}^{2}+1)^{-1}(\alpha \cross_m \dot{\alpha})\rangle_m\\
&= (4\pi^{2}(1+r^{2})+1)^{-1}\int_{S^{1}} 
\langle \lambda_1\dot{\alpha}+\lambda_2 (\alpha \cross_m \dot{\alpha}), 
\alpha \cross_m \dot{\alpha}\rangle_m=0.
\end{align*}
Consequently, $D_{g_0}X_{k_0,g_0,\hz}\eval_{\alpha}(E_{+})=E_{+}$.\\
$E_+$ is $L^{2}$-orthogonal to $\alpha \times_m \dot\alpha$ and $\dot\alpha$, we may thus write
\begin{align*}
V &= (\nu_1+f_1)\dot{\alpha}+(\nu_2+f_2)(\alpha \cross_m \dot{\alpha}),  
\end{align*}
where 
\begin{align*}
\nu_1,\nu_2 \perp_{L^2} \langle 1, \sin(2\pi \cdot),\cos(2\pi \cdot)\rangle 
\text{ and }
f_1,f_2 \in \langle \sin(2\pi \cdot),\cos(2\pi \cdot)\rangle.  
\end{align*}
Then 
\begin{align}
\label{eq:5}
\langle (-D_{t,g_0}^{2}&+1) D_{g_0}X_{k_0,g_0,\hz}\eval_{\alpha}(V),V\rangle_{L^{2}} \notag \\
&=
\int_{S^{1}} (\nu_1')^2-2\pi \sqrt{1+r^2}\nu_1'\nu_2+(\nu_2')^2-4\pi^2(\nu_2)^2 \notag \\
&\quad \int_{S^{1}} (f_1')^2-2\pi \sqrt{1+r^2}f_1'f_2.
\end{align}
Since $\nu_2 \perp_{L^2} \langle 1,\, \cos(2\pi\cdot),\sin(2\pi\cdot)\rangle$
we have
\begin{align*}
\int_{S^{1}}(\nu_2')^{2}-4\pi^{2}(\nu_2)^{2}\ge \int_{S^{1}}16\pi^{2}(\nu_2)^{2}  
\end{align*}
and for $0<r\le 1$
\begin{align*}
\int_{S^{1}} (\nu_1')^2&-2\pi \sqrt{1+r^2}\nu_1'\nu_2+(\nu_2')^2-4\pi^2(\nu_2)^2\\
&\ge 
\int_{S^{1}} (\nu_1')^2-\frac{1}{4}(\nu_1')^2-4\pi^2(1+r^2)(\nu_2)^2
+(\nu_2')^2-4\pi^2(\nu_2)^2\\
&\ge 
\int_{S^{1}} \frac{3}{4}(\nu_1')^2+4\pi^2(\nu_2)^2.  
\end{align*}
Concerning the remaining term in \eqref{eq:5} we note that
as $(-D_{t,g_0}^{2}+1)$ maps 
\begin{align*}
\big\{\lambda_1 \dot{\alpha}+ \lambda_2 (\alpha \cross_m \dot{\alpha}) \where
\lambda_1,\lambda_2 \in \langle \sin(2\pi \cdot),\cos(2\pi \cdot)\rangle 
\big\}
\end{align*}
into itself and $V \in E_{+}$ there holds 
\begin{align*}
f_1 \dot{\alpha}+f_2(\alpha \cross_m \dot{\alpha})
\in (-D_{t,g_0}^{2}+1)^{-1}
\big\langle \cos(2\pi \cdot) \dot\alpha,\sin(2\pi \cdot) \dot\alpha\big\rangle. 
\end{align*}
Hence, by explicit computations there are $x,y \in \rz$ satisfying
\begin{align*}
f_1(t) &= x \cos(2\pi t)+ y \sin(2\pi t),\\
f_2(t) &= \frac{8\pi^2\sqrt{1+r^2}}{4\pi^2(2+r^2)+1}
\big(
y\cos(2\pi t)-x\sin(2\pi t)
\big),  
\end{align*}
such that
\begin{align*}
\int_{S^{1}} (f_1')^2&-2\pi \sqrt{1+r^2}f_1'f_2
= \frac{2\pi^2(1-4\pi^2r^2)}{4\pi^2(2+r^2)+1}(x^2+y^2).
\end{align*}
This shows that if $r < (2\pi)^{-1}$, then
\begin{align*}
\langle (-D_{t,g_0}^{2}+1) D_{g_0}X_{k_0,g_0,\hz}\eval_{\alpha}(V),V\rangle_{L^{2}} >0
\text{ for all } V \in E_{+}\setminus \{0\},  
\end{align*}
and the homotopy
\begin{align*}
[0,1] \ni s \mapsto (1-s) \big(D_{g_0}X_{k_0,g_0,\hz}\eval_{\alpha}\big)\eval_{E_{+}} +s\, id\eval_{E_{+}}  
\end{align*}
is admissible.
We use the decomposition in \eqref{eq:4} and
\begin{align*}
D_{g_0}X_{k_0,g_0,\hz}\eval_{\alpha}(\alpha \cross_m \dot{\alpha}) = -\frac{4 \pi^2}{4\pi^2(1+r^2)+1} 
(\alpha \cross_m \dot{\alpha})   
\end{align*}
to see that under the assumption $r < (2\pi)^{-1}$
\begin{align*}
\big(D_{g_0}X_{k_0,g_0,\hz}\eval_{\alpha}\big)\eval_{R(D_{g_0}X_{k_0,g_0,\hz}\eval_{\alpha})} 
\sim 
\begin{pmatrix}
-1 & 0\\
0 &   id\eval_{E_{+}}
\end{pmatrix}.
\end{align*} 
Consequently, for $r < (2\pi)^{-1}$
\begin{align}
\label{eq:sgn_dx_g0_range}
\sgn \big(D_{g_0}X_{k_0,g_0,\hz}\eval_{\alpha}\big)\eval_{R(D_{g_0}X_{k_0,g_0,\hz}\eval_{\alpha})} =-1.
\end{align}
We remark that
the formula for the degree continues to hold for $r > (2\pi)^{-1}$.\\
From \eqref{eq:3} and the fact that $(\pi_M)_*$ is an isometry we obtain
for $\tilde{\alpha} \in \mathcal{Z}_{M}$
\begin{align}
\label{eq:6}
\{0\}&=\langle \tilde{W}_1,\tilde{W}_2,\tilde{W}_3\rangle 
\cap R\big(D_{g_0}X_{k_0,g_0}\eval_{\tilde{\alpha}}\big) ,\notag \\
\langle \tilde{W}_1\rangle^{\perp}&=  
\langle \tilde{W}_2,\tilde{W}_3\rangle \oplus R\big(D_{g_0}X_{k_0,g_0}
\eval_{\tilde{\alpha}}\big), 
\notag\\
-1 &= sgn \big(D_{g_0}X_{k_0,g_0}\eval_{\tilde{\alpha}}\big)
\eval_{R(D_{g_0}X_{k_0,g_0}\eval_{\tilde{\alpha}})}.
\end{align}
We fix $\tilde{\alpha}_0 \in \mathcal{Z}_{M}$ and a parametrization $\phi$ of $\mathcal{Z}_{M}$, 
which maps an open neighborhood of 
$0$ in $\langle \tilde{W}_1(\tilde{\alpha}_0),\tilde{W}_2(\tilde{\alpha}_0),\tilde{W}_3(\tilde{\alpha}_0)\rangle$ into $\mathcal{Z}_{M}$, such that
\begin{align*}
\phi(0)=\tilde{\alpha}_0 \text{ and } D\phi\eval_0=id.  
\end{align*}
As $\mathcal{Z}_{M}$ consists of smooth functions, 
$\mathcal{Z}_{M}$ is a sub-manifold of $H^{m,2}(S^{1},M)$ for $1\le m<\infty$.
We define a map $\Phi$ from an open neighborhood $\mathcal{U}$
of $0$ in
\begin{align*}
T_{\tilde{\alpha}_0}H^{2,2}(S^{1},M)=\langle \tilde{W}_1(\tilde{\alpha}_0),\tilde{W}_2(\tilde{\alpha}_0),\tilde{W}_3(\tilde{\alpha}_0)\rangle \oplus 
\text{Range}(D_{g_0}X_{k_0,g_0}\eval_{\tilde{\alpha}_0})  
\end{align*}
to $H^{2,2}(S^{1},M)$ by 
\begin{align}
\label{eq:15}
\Phi(W,U):= Exp_{\tilde{\alpha}_0,g_0}\big(Exp_{\tilde{\alpha}_0,g_0}^{-1}(\phi(W))+U\big).  
\end{align}
Then $(\Phi,\mathcal{U})$ is a chart
of $H^{2,2}(S^{1},M)$ around $\tilde{\alpha}_0$ such that
$\mathcal{U}$ is an open neighborhood of $0$ in $T_{\tilde{\alpha}_0}H^{2,2}(S^{1},M)$, 
$\Phi(0)=\tilde{\alpha}_0$, and
\begin{align*}
 D\Phi\eval_0=id,\; 
\Phi^{-1}\big(\mathcal{Z}_{M}\cap \Phi(\mathcal{U})\big)= \mathcal{U}\cap \langle \tilde{W}_1(\tilde{\alpha}_0),\tilde{W}_2(\tilde{\alpha}_0),\tilde{W}_3(\tilde{\alpha}_0)\rangle.
\end{align*}
From the properties of $Exp_{\tilde{\alpha}_0,g_0}$ the map $\Phi$ is a chart of $H^{k,2}(S^{1},M)$ 
around
$\tilde{\alpha}_0$ for any $1\le k \le 4$ and shrinking $\mathcal{U}$ we may assume that
\begin{align}
\label{eq:def_trans_1_phi}
T_{\Phi(V)}H^{1,2}(S^{1},M) 
&= \langle \frac{d}{dt}\Phi(V)\rangle \oplus D\Phi\eval_V(\langle \dot\tilde{\alpha}_0\rangle^{\perp,H^{1,2}}),\\
\label{eq:def_trans_2_phi}
T_{\Phi(V)}H^{2,2}(S^{1},M) 
&= \langle W_{g_0}(\Phi(V))\rangle \oplus D\Phi\eval_V(\langle W_{g_0}(\tilde{\alpha}_0)\rangle^{\perp}),\\
\label{eq:def_trans_3_phi}
\text{Proj}_{\langle W_{g_0}(\Phi(V)\rangle^\perp}\circ D \Phi\eval_{V}&:\,
\langle W_{g_0}(\tilde{\alpha}_0)\rangle^\perp \xpfeil{\iso}{} \langle W_{g_0}(\Phi(V)\rangle^\perp,
\end{align}
and the norm of the projections in (\ref{eq:def_trans_1_phi}) and (\ref{eq:def_trans_2_phi})
as well as the norm of the map in (\ref{eq:def_trans_3_phi}) 
and its inverse are uniformly bounded with respect to $V$.

\section{The perturbative analysis}
\label{sec:perturbed}
For $\tilde{\alpha}_0 \in \mathcal{Z}_{M}$ 
the vectors $\tilde{W}_1(\tilde{\alpha}_0)$ and $W_{g_0}(\tilde{\alpha}_0)$ are collinear
and we use $\langle \tilde{W}_1(\tilde{\alpha}_0)\rangle$ 
instead of $\langle W_{g_0}(\tilde{\alpha}_0)\rangle$
in the analysis below.\\
We define a $S^{1}$-invariant 
vector bundle $SH^{2,2}(S^{1},M)$ by
\begin{align*}
SH^{2,2}(S^{1},M) := \{(\gamma,V)\in TH^{2,2}(S^{1},M)
\where \gamma \neq {\rm const},\,V\in \langle W_g(\gamma)\rangle^\perp\}.   
\end{align*}
As in \cite[Sec. 4]{arXiv:0808.4038} we obtain
a chart $\Psi$ for the bundle $SH^{2,2}(S^{1},M)$ around $(\tilde{\alpha}_0,0)$ by,
\begin{align*}
\Psi: \mathcal{U}\times \mathcal{U}\cap \langle \tilde{W}_1(\tilde{\alpha}_0)\rangle^{\perp}
\to SH^{2,2}(S^{1},M),\\
\Psi(V,U):= \big(\Phi(V),Proj_{\langle W_{g_0}(\Phi(V))\rangle^{\perp}}\circ D\Phi\eval_V(U)\big).
\end{align*}
We define
\begin{align*}
X_{g_0,\eps}^{\Phi}:\: \mathcal{U}\cap \langle \tilde{W}_1(\tilde{\alpha}_0)\rangle^{\perp}
\to \langle \tilde{W}_1(\tilde{\alpha}_0)\rangle^{\perp}   
\end{align*}
by
\begin{align*}
X_{g_0,\eps}^{\Phi}(V) := Proj_2 \circ \Psi^{-1}\big(\Phi(V),X_{g_0,\eps}(\Phi(V))\big).  
\end{align*}
As in \cite[Lem. 3.5]{arXiv:0808.4038} it is easy to see that
\begin{align}
\label{eq:nondegenerate_phi}
V \in \mathcal{U}\cap \langle \tilde{W}_1(\tilde{\alpha}_0)\rangle^{\perp}
\text{ is a (nondegenerate) zero of } X_{g_0,\eps}^{\Phi} \text{ if and only if } \notag \\
S^{1}*\Phi(V) \text{ is a (nondegenerate) critical orbit of } X_{g_0,\eps},
\end{align}
and
if $X_{g_0,\eps}^{\Phi}(V)=0$, then
\begin{align}
\label{eq:d_x_phi}
D_{g_0}X_{g_0,\eps}^{\Phi}\eval_V = A_V^{-1}\circ D_{g_0}X_{g_0,\eps}\eval_{\Phi(V)} \circ D\Phi\eval_V,  
\end{align}
where the isomorphism $A_V:\: \langle \tilde{W}_1(\tilde{\alpha}_0)\rangle^{\perp} \to \langle W_{g_0}(\Phi(V))\rangle^{\perp}$ 
is given by
\begin{align*}
A_V = Proj_{\langle W_{g_0}(\Phi(V))\rangle^{\perp}}\circ D\Phi\eval_V.  
\end{align*} 
From Lemma \ref{l:kernel_not_in_range} we may assume
\begin{align*}
\mathcal{U}\cap \langle \tilde{W}_1(\tilde{\alpha}_0)\rangle^{\perp} = \mathcal{U}_1\times \mathcal{U}_2,   
\end{align*}
where $\mathcal{U}_1$ and $\mathcal{U}_2$ are open neighborhoods of $0$ 
in $\langle \tilde{W}_2(\tilde{\alpha}_0),\tilde{W}_3(\tilde{\alpha}_0)\rangle$
and $R\big(D_{g_0}X_{k_0,g_0}\eval_{\tilde{\alpha}_0}\big)$.  
We denote for $\tilde{\alpha} \in \mathcal{Z}_{M}$ by $P_2(\tilde{\alpha})$ the projection onto 
$R(D_{g_0}X_{g_0,0}\eval_{\tilde{\alpha}})$ with respect to the decomposition
\begin{align*}
\langle \tilde{W}_1(\tilde{\alpha})\rangle^{\perp}&=  
\langle \tilde{W}_2(\tilde{\alpha}),\tilde{W}_3(\tilde{\alpha})\rangle \oplus R\big(D_{g_0}X_{k_0,g_0}\eval_{\tilde{\alpha}}\big),
\end{align*}
and by $P_1(\tilde{\alpha})$ the projection onto $\langle \tilde{W}_2(\tilde{\alpha}),\tilde{W}_3(\tilde{\alpha})\rangle$. 
Moreover, for $W \in \mathcal{U}_1$ we define for $i=1,2$
\begin{align*}
P_i^{\Phi}(W):= (A_W)^{-1}\circ P_i(\Phi(W)) \circ A_W.
\end{align*}
The projections $P_1^{\Phi}(W)$ and $P_2^{\Phi}(W)$ correspond to the decomposition
\begin{align}
\label{eq:decomp_phi_w}
\langle \tilde{W}_1(\tilde{\alpha}_0)\rangle^{\perp}&=  
\langle \tilde{W}_2(\tilde{\alpha}_0),\tilde{W}_3(\tilde{\alpha}_0)\rangle \oplus R\big(D_{g_0}X_{g_0,0}^{\Phi}\eval_{W}\big),
\end{align}
as we have for $W \in \mathcal{U}_1$
\begin{align*}
D_{g_0}X_{g_0,0}^{\Phi}\eval_W = A_W^{-1}\circ D_{g_0}X_{g_0,0}\eval_{\Phi(W)} \circ A_W.  
\end{align*}
Moreover, for $\tilde{\alpha} \in \mathcal{Z}_{M}$ the vector field $K_1(\tilde{\alpha})$ is
orthogonal to $\tilde{W}_1(\tilde{\alpha})$ and we may define a vector field on $\mathcal{Z}_{M}$ by
\begin{align*}
\mathcal{Z}_{M} \ni \tilde{\alpha} \mapsto P_1(\tilde{\alpha}) \circ K_1(\tilde{\alpha}) \in 
\langle \tilde{W}_2(\tilde{\alpha}),\tilde{W}_3(\tilde{\alpha})\rangle.   
\end{align*}
Note that $P_1(\cdot) \circ K_1(\cdot)$ is $S^1$-equivariant, i.e.
\begin{align*}
\theta * \big(P_1(\tilde{\alpha}) \circ K_1(\tilde{\alpha})\big) 
= P_1(\theta*\tilde{\alpha}) \circ K_1(\theta*\tilde{\alpha}) \text{ for all }(\theta,\tilde{\alpha})
\in S^1\times \mathcal{Z}_{M}.  
\end{align*}
If $P_1(\tilde{\alpha}_0) \circ K_1(\tilde{\alpha}_0) =0$ for some $\tilde{\alpha}_0 \in \mathcal{Z}_{M}$
differentiating the identity 
\begin{align*}
0 \equiv 
\langle P_1(\tilde{\alpha}) \circ K_1(\tilde{\alpha}), \tilde{W}_1(\tilde{\alpha})\rangle  
\end{align*}
we find that the covariant derivative 
$D_{\mathcal{Z}_{M}}\big(P_1(\cdot) \circ K_1(\cdot)\big)\eval_{\tilde{\alpha}_0}$
maps 
\begin{align*}
T_{\tilde{\alpha}_0}\mathcal{Z}_{M} 
= \langle \tilde{W}_1(\tilde{\alpha}_0),\tilde{W}_2(\tilde{\alpha}_0),
\tilde{W}_3(\tilde{\alpha}_0)\rangle  
\end{align*}
to $\langle \tilde{W}_2(\tilde{\alpha}_0),\tilde{W}_3(\tilde{\alpha}_0)\rangle$
and the $S^1$ equivariance leads to
\begin{align*}
D_{\mathcal{Z}_{M}}\big(P_1(\cdot) \circ K_1(\cdot)\big)\eval_{\tilde{\alpha}_0}
\big(\tilde{W}_1(\tilde{\alpha}_0)\big) =0.  
\end{align*}
Consequently, we say that $S^1*\tilde{\alpha}_0 \in \mathcal{Z}_{M}$ is a nondegenerate zero orbit
of $P_1(\cdot) \circ K_1(\cdot)$, if $P_1(\tilde{\alpha}_0) \circ K_1(\tilde{\alpha}_0)=0$ and
\begin{align*}
D_{\mathcal{Z}_{M}}\big(P_1(\cdot) \circ K_1(\cdot)\big)\eval_{\tilde{\alpha}_0}:\:
\langle \tilde{W}_2(\tilde{\alpha}_0),\tilde{W}_3(\tilde{\alpha}_0)\rangle \to \langle \tilde{W}_2(\tilde{\alpha}_0),\tilde{W}_3(\tilde{\alpha}_0)\rangle  
\end{align*}
is invertible.\\
Using the above notation the perturbative analysis done in \cite{arXiv:0808.4038}
carries over and we state the following four results without proof
(see \cite[Lem. 5.2-5.5]{arXiv:0808.4038}).  
\begin{lemma}
\label{l:implicit_function}
For $\tilde{\alpha}_0\in \mathcal{Z}$ after possibly shrinking $\mathcal{U}$
there are $\eps_0>0$ and 
\begin{align*}
U &\in C^2([-\eps_0,\eps_0]\times \mathcal{U}_1,\langle \tilde{W}_1(\tilde{\alpha}_0)\rangle^\perp),\\
R &\in C^2([-\eps_0,\eps_0]\times \mathcal{U}_1,\langle \tilde{W}_2(\tilde{\alpha}_0),\tilde{W}_3(\tilde{\alpha}_0)\rangle),
\end{align*}
such that for all $(\eps,W) \in [-\eps_0,\eps_0]\times \mathcal{U}_1$ 
\begin{align*}
R(\eps,W) &=X_{g_0,\eps}^{\Phi}(W+U(\eps,W)),\\
0&= P_1^\Phi(W)\circ U(\eps,W),\\
O(\eps)_{\eps \to 0} &=\|U(\eps,W)\|+\|D_W U(\eps,W)\|+\|R(\eps,W)\|+\|D_W R(\eps,W)\|,\\
R(\eps,W)&= \eps P_1^\Phi(W)\circ K_1^{\Phi}(W)+o(\eps)_{\eps \to 0},\\
U(\eps,W) &= -\eps (D_{g_0}X_{g_0,0}^\Phi\eval_{W})^{-1}\circ  P_2^\Phi(W)\circ K_1^{\Phi}(W)+o(\eps)_{\eps \to 0}.  
\end{align*}
Moreover, the functions $U(\eps, W)$ and $R(\eps,W)$ are unique, in the sense that, 
if $(\eps,W,U,R)$ in $[-\eps_0,\eps_0]\times \mathcal{U}_1\times \mathcal{U}\cap\langle \tilde{W}_1(\tilde{\alpha}_0)\rangle^\perp
\times \mathcal{U}_1$ 
satisfies
\begin{align*}
X_{g_0,\eps}^{\Phi}(W+U)=R \text{ and }
P_1^\Phi(W)\big(U\big) =0,  
\end{align*}
then $U=U(\eps,W)$ and $R=R(\eps,W)$.
\end{lemma}

\begin{lemma}
\label{l:expansion_eps_0}
Under the assumptions of Lemma \ref{l:implicit_function} we have
as $\eps \to 0$
\begin{align*}
X_{g_0,\eps}^{\Phi}(W+U(\eps,W))
&= \eps P_1^\Phi(W)\circ K_1^{\Phi}(W) + O(\eps^{2})_{\eps \to 0},  
\end{align*}
where $K_1^{\Phi}$ is the vector-field $K_1$ in the coordinates $\Phi$, i.e.
\begin{align*}
K_1^{\Phi} = X_{g_0,1}^\Phi-X_{g_0,0}^\Phi.  
\end{align*}
\end{lemma}

\begin{lemma}
\label{l:nondeg_zero}
Under the assumptions of Lemma \ref{l:implicit_function}
suppose $0$ is a nondegenerate zero of the vector-field $P_1^\Phi(\cdot)\circ K_1^\Phi(\cdot)$,
in the sense that $P_1^\Phi(0)\circ K_1^\Phi(0)=0$
and
\begin{align*}
D_W(P_1^\Phi(\cdot)\circ K_1^\Phi(\cdot))\eval_0 
\in \mathcal{L}(\langle {W}_2(\tilde{\alpha}_0),{W}_3(\tilde{\alpha}_0)\rangle)  
\end{align*}
is an isomorphism. Then, after possibly shrinking $\eps_0$ and $\mathcal{U}$,
for any $0<\eps\le \eps_0 $ there is a unique $W(\eps)\in \mathcal{U}_1$
such that
\begin{align*}
X_{g_0,\eps}^{\Phi}(W(\eps)+U(\eps,W(\eps)))=0,\\
W(\eps) \to 0 \text{ as } \eps \to 0.  
\end{align*}
Moreover, $V(\eps):= W(\eps)+U(\eps,W(\eps))$ is the only zero of
$X_{g_0,\eps}^{\Phi}$ in $\mathcal{U}\cap \langle \tilde{W}_1(\tilde{\alpha}_0)\rangle^{\perp}$
and is nondegenerate with
\begin{align*}
\sgn(D X_{g_0,\eps}^{\Phi}\eval_{V(\eps)}) 
&= -det(D_W(P_1^\Phi(\cdot)\circ K_1^\Phi(\cdot))\eval_0).  
\end{align*}
\end{lemma}

\begin{lemma}
\label{l:nondeg_zero_gamma}
Under the assumptions of Lemma \ref{l:implicit_function}
suppose $\tilde{\alpha}_0$ is a nondegenerate zero of the vector-field $P_1(\cdot)\circ K_1(\cdot)$ on $\mathcal{Z}_{M}$,
in the sense that $P_1(\tilde{\alpha}_0)\circ K_1(\tilde{\alpha}_0)=0$
and
\begin{align*}
D_{\mathcal{Z}}(P_1(\cdot)\circ K_1(\cdot))\eval_{\tilde{\alpha}_0} 
\in \mathcal{L}(\langle {W}_2(\tilde{\alpha}_0),{W}_3(\tilde{\alpha}_0)\rangle) 
\end{align*}
is an isomorphism. Then 
for any $0<\eps<\eps_0$ there is $\tilde{\gamma}(\eps) \in \Phi(\mathcal{U})$ satisfying
\begin{align*}
X_{g_0,\eps}(\tilde{\gamma}(\eps))=0 \text{ and }
\tilde{\gamma}(\eps) \to \tilde{\alpha}_0 \text{ as } \eps \to 0.  
\end{align*}
Moreover, $S^1*\tilde{\gamma}(\eps)$ is the unique critical orbit of $X_{g_0,\eps}$ in $\Phi(\mathcal{U})$
and is nondegenerate with
\begin{align*}
\deg_{loc, S^{1}}(X_{g_0,\eps},S^1*\tilde{\gamma}(\eps)) &= -\det(D_{\mathcal{Z}_{M}}(P_1(\cdot)\circ K_1(\cdot))\eval_{\tilde{\alpha}_0}).  
\end{align*}
\end{lemma}

\section{The computation of the degree}
\label{sec:computation_degree}
In order to compute the $S^{1}$-degree of $X_{g_0,\eps}$ we choose a smooth Morse function 
$k_1 \in C^{\infty}(M,\rz)$.
The corresponding vector-field $K_1$ on $H^{2,2}(S^{1},M)$ is given by
\begin{align*}
K_1(\gamma) = (-D_{t,g_0}^{2}+1)^{-1}(|\dot \gamma|_{g_0} k_1(\tilde{\gamma}) J_{g_0}(\gamma)\dot\gamma).  
\end{align*}
We note that for $\tilde{\alpha} = \pi_M \circ \alpha(\cdot,2\pi|r|,v_0,v_1,w) \in \mathcal{Z}_{M}$ 
and $r>0$ small enough we have
\begin{align*}
K_1(\tilde{\alpha}) = 2\pi r
(-D_{t,g_0}^{2}+1)^{-1}
\big(
k_1(\tilde{\alpha}) (\phi_M)_*(\alpha \times_m \dot\alpha)
\big).
\end{align*}
Consequently, from \eqref{eq:D2_1_w2}, \eqref{eq:D2_1_w3}, and \eqref{eq:4}
\begin{align}
\label{eq:11}
P_1(\tilde{\alpha})\circ K_1(\tilde{\alpha}) = 
\sigma_2(\tilde{\alpha}) \tilde{W}_2(\tilde{\alpha}) 
+ \sigma_3(\tilde{\alpha}) \tilde{W}_3(\tilde{\alpha}),   
\end{align}
where $\sigma_2(\tilde{\alpha}),\, \sigma_2(\tilde{\alpha}) \in \rz$ are defined by the condition that
\begin{align*}
2\pi r k_1(\tilde{\alpha}) 
-\frac{\sigma_2(\tilde{\alpha})}{2\pi r}   
(1-4\pi^{2}r^{2})\cos(2\pi \cdot)
-\frac{\sigma_3(\tilde{\alpha})}{2\pi r}   
(1-4\pi^{2}r^{2})\sin(2\pi \cdot) 
\end{align*}
is $L^{2}$-orthogonal to $\langle \cos(2\pi \cdot), \sin(2\pi \cdot)\rangle$.
Hence,
\begin{align*}
\sigma_2(\tilde{\alpha}) &= \frac{8\pi^{2} r^{2}}{1-4\pi^{2}r^{2}}
\int^{1}_{0} k_1\circ \tilde{\alpha}(t) \cos(2\pi t) \, dt,\\
\sigma_3(\tilde{\alpha}) &= \frac{8\pi^{2} r^{2}}{1-4\pi^{2}r^{2}}
\int^{1}_{0} k_1\circ \tilde{\alpha}(t) \sin(2\pi t) \, dt.
\end{align*}
In the following we are interested in the asympotics of $\sigma_2$ and
$\sigma_3$ as $r \to 0^+$ or equivalently as $k_0 \to \infty$.
There holds
\begin{align}
\label{eq:7}
&\frac{1-4\pi^{2}r^{2}}{8\pi^{2} r^{2}}\sigma_2(\tilde{\alpha}) \notag \\
&\quad = \int^{1}_{0} 
\Big( k_1\circ \pi_M(w) 
+ r dk_1\eval_{\pi_M(w)} \cos(2\pi t)(\pi_M)_* v_1 \notag \\
&\qquad +r dk_1\eval_{\pi_M(w)}\sin(2\pi t)(\pi_M)_* v_0 + O(r^2)\Big)
\cos(2\pi t) \, dt \notag\\
&\quad = \frac{1}{2} r dk_1\eval_{\pi_M(w)} (\pi_M)_* v_1 + O(r^2), 
\end{align}
and analogously we find
\begin{align}
\label{eq:9}
\frac{1-4\pi^{2}r^{2}}{8\pi^{2} r^{2}}\sigma_3(\tilde{\alpha})
&=  \frac{1}{2} r dk_1\eval_{\pi_M(w)} (\pi_M)_* v_0 + O(r^2). 
\end{align}
From the above expansion we easily deduce
\begin{lemma}
\label{l:k_1_1}
For all $\delta>0$ there is $r_0>0$ such that for all $0<r\le r_0$ and
\begin{align*}
\tilde{\alpha}= \pi_M(\sqrt{1+r^2}w+r\cos(2\pi t)v_1+r\sin(2\pi t)v_0) \in \mathcal{Z}_{M}  
\end{align*}
satisfying $P_1(\tilde{\alpha})\circ K_1(\tilde{\alpha}) =0$ there holds
\begin{align*}
\pi_M(w) \in \cupl_{i=1}^n B_{\delta}(\tilde{w}_i),  
\end{align*}
where $\{\tilde{w}_i\where 1\le i \le n\}$ denotes the set of critical points of $k_1$
in $M$.
\end{lemma}
Fix $w_0 \in \hz$ and a positive orthonormal system $\{v_0,v_1,w_0\}$ in $(\rz^3,m)$ such that 
$\pi_M(w_0)$ is a critical point of $k_1$ in $M$. We choose $\delta>0$, a parametrization
\begin{align*}
w:\: B_1(0) \subset \rz^2 \to B_{\delta}(w_0) \subset \hz,\;
(x,y) \mapsto w(x,y),  
\end{align*}
and smooth maps $v_0, v_1:\: B_1(0) \to \rz^3$ such that
$\{v_0(x,y),v_1(x,y),w(x,y)\}$ is orthonormal for all $(x,y) \in B_1(0)$ 
and
\begin{align}
\label{eq:10}
(v_0(0,0),v_1(0,0),w(0,0))=(v_0,v_1,w_0),\,
\frac{\rand w}{\rand x}\eval_{(0,0)}=v_1,\, \frac{\rand w}{\rand y}\eval_{(0,0)}=v_0.
\end{align}
Shrinking $\delta>0$ we may assume that 
$\pi_M \circ \phi_w$ parametrizes $M$ and that
$(x,y) \mapsto \tilde{\alpha}(x,y)$
is an injective immersion from $B_1(0)$ to $\mathcal{Z}_M$, where
\begin{align*}
\tilde{\alpha}(x,y) := 
\pi_M \big(\sqrt{1+r^2} w(x,y)+r \cos(2\pi \cdot) v_1(x,y) +r \sin(2\pi \cdot) v_0(x,y)\big).
\end{align*}
From \eqref{eq:12} and \eqref{eq:10} we get as $r,\delta \to 0^+$
\begin{align}
\label{eq:13}
\frac{\rand}{\rand x} \tilde{\alpha}\eval_{(x,y)}
&= \tilde{W}_2(\tilde{\alpha}(x,y)) + O(r)+O(\delta) \notag\\
\frac{\rand}{\rand y} \tilde{\alpha}\eval_{(x,y)}
&= \tilde{W}_3(\tilde{\alpha}(x,y)) + O(r)+O(\delta).
\end{align}
Define $H:\: B_1(0) \to \rz^2$ by
\begin{align*}
H(x,y) := \big(\sigma_2(\tilde{\alpha}(x,y)),\sigma_3(\tilde{\alpha}(x,y))\big).  
\end{align*}
By \eqref{eq:7} and \eqref{eq:9} we have as $r \to 0^+$
\begin{align*}
\frac{H(x,y)}{8\pi^2r^3} &:=  
\big(dk_1\eval_{\pi_M(w(x,y))} (\pi_M)_* v_1(x,y), 
dk_1\eval_{\pi_M(w(x,y))} (\pi_M)_* v_0(x,y)\big)\\
&\quad +O(r).  
\end{align*}
Since
\begin{align*}
\frac{d}{dx} \pi_M \circ w\eval_{0,0} = (\pi_M)_* v_1(0,0),\;
\frac{d}{dy} \pi_M \circ w\eval_{0,0} = (\pi_M)_* v_0(0,0)  
\end{align*}
we find for small values of $\delta>0$ and $r>0$
\begin{align*}
\deg(H,B_1(0),0)&= \deg(\nabla (k_1\circ \pi_M \circ w),B_1(0),0)\\
&= \deg(\nabla k_1,B_{\delta}(\pi_M(w_0)),0)
= \deg_{loc}(\nabla k_1,\pi_M(w_0)),  
\end{align*}
and the set of zeros of $H$ in $B_1(0)$ is non-empty. Fix a zero $(x_0,y_0) \in B_1(0)$ of $H$.
Then
\begin{align*}
dH\eval_{(x_0,y_0)} =
\begin{pmatrix}
\frac{\rand}{\rand x}(\sigma_2 \circ \tilde{\alpha})\eval_{(x_0,y_0)} &
\frac{\rand}{\rand y}(\sigma_2 \circ \tilde{\alpha})\eval_{(x_0,y_0)}\\  
\frac{\rand}{\rand x}(\sigma_3 \circ \tilde{\alpha})\eval_{(x_0,y_0)} &
\frac{\rand}{\rand y}(\sigma_3 \circ \tilde{\alpha})\eval_{(x_0,y_0)}\\  
\end{pmatrix}.
\end{align*}
From \eqref{eq:7}, \eqref{eq:9}, and the fact that $H(x_0,y_0)=0$ we get 
\begin{align*}
dk_1\eval_{\pi_M(w(x_0,y_0))} = O(r).
\end{align*}
Thus, we have as $r \to 0^+$
\begin{align*}
dk_1&\eval_{\tilde{\alpha}(x_0,y_0)(t)} \frac{\rand}{\rand x}\tilde{\alpha}\eval_{(x_0,y_0)}(t)\\
&= dk_1\eval_{\pi_M(w(x_0,y_0))} (\pi_M)_* \frac{\rand }{\rand x}w\eval_{(x_0,y_0)}\\
&\quad + r \Big(\nabla_{\frac{\rand}{\rand r} \tilde{\alpha}(x_0,y_0)(t)\eval_{r=0}} 
dk_1\eval_{\pi_M(w(x_0,y_0))} (\pi_M)_*\frac{\rand }{\rand x}w\eval_{(x_0,y_0)}\\
&\quad +dk_1\eval_{\pi_M(w(x_0,y_0))} 
\nabla_{\frac{\rand}{\rand r} \tilde{\alpha}(x_0,y_0)(t)\eval_{r=0}}
\frac{\rand}{\rand x}\tilde{\alpha}(x_0,y_0)(t)\eval_{r=0}\Big)
+O(r^2)\\
&= dk_1\eval_{\pi_M(w(x_0,y_0))} (\pi_M)_* \frac{\rand }{\rand x}w\eval_{(x_0,y_0)}\\
&\quad + r (\nabla dk_1)\eval_{\pi_M(w(x_0,y_0))}
\Big((\pi_M)_* \frac{\rand }{\rand x}w\eval_{(x_0,y_0)},\\
&\qquad\cos(2\pi t) (\pi_M)_* v_1(x_0,y_0)+\sin(2\pi t) (\pi_M)_* v_0(x_0,y_0)\Big)
+O(r^2).
\end{align*}
Using \eqref{eq:10}, this leads to, as $r,\delta \to 0^{+}$
\begin{align*}
\frac{\rand}{\rand x}&(\sigma_2 \circ \tilde{\alpha})\eval_{(x_0,y_0)}\\
&= \int_0^1 dk_1\eval_{\tilde{\alpha}(x_0,y_0)(t)} 
\frac{\rand}{\rand x}\tilde{\alpha}\eval_{(x_0,y_0)}(t)
\cos(2\pi t) dt\\
&=  \frac{r}{2} (\nabla dk_1)\eval_{\pi_M(w(x_0,y_0))}
((\pi_M)_* \frac{\rand }{\rand x}w\eval_{(x_0,y_0)}, (\pi_M)_* v_1(x_0,y_0))
+O(r^2)\\
&=  \frac{r}{2} (\nabla dk_1)\eval_{\pi_M(w_0)}
\Big((\pi_M)_* v_1, (\pi_M)_* v_1\Big)
+O(r^2)+O(r\delta).
\end{align*}
Analogously, we may compute the remaining partial derivatives of $H$
and we find for small values of $\delta>0$ and $r>0$
\begin{align}
\label{eq:14}
\sgn \det(dH\eval_{(x_0,y_0)}) =
\sgn \det(\nabla dk_1\eval_{\pi_M(w_0)})=\deg_{loc}(\nabla k_1,\pi_M(w_0)),    
\end{align}
such that $(x_0,y_0)$ is the unique zero of $H$ in $B_1(0)$. 
From \eqref{eq:11} we see that 
\begin{align*}
P_1(\tilde{\alpha}(x_0,y_0)) \circ K_1(\tilde{\alpha}(x_0,y_0)) =0,
\end{align*}
by \eqref{eq:13} we obtain as $r,\delta \to 0^+$
\begin{align*}
&\nabla_{\tilde{W}_2}\big(P_1(\cdot) \circ K_1(\cdot)\big)\eval_{\tilde{\alpha}(x_0,y_0)}\\
&\quad= \Big(d\sigma_2\eval_{\tilde{\alpha}(x_0,y_0)} \tilde{W}_2(\tilde{\alpha}(x_0,y_0))\Big)
\tilde{W}_2(\tilde{\alpha}(x_0,y_0))\\
&\qquad +
\Big(d\sigma_3\eval_{\tilde{\alpha}(x_0,y_0)} \tilde{W}_2(\tilde{\alpha}(x_0,y_0))\Big)
\tilde{W}_3(\tilde{\alpha}(x_0,y_0))\\
&\quad = \frac{\rand}{\rand x}(\sigma_2 \circ \tilde{\alpha})\eval_{(x_0,y_0)}
\tilde{W}_2(\tilde{\alpha}(x_0,y_0))\\
&\qquad +
\frac{\rand}{\rand x}(\sigma_3 \circ \tilde{\alpha})\eval_{(x_0,y_0)}
\tilde{W}_3(\tilde{\alpha}(x_0,y_0)) + O(r)+O(\delta).  
\end{align*}
Concerning the covariant derivative of $P_1(\cdot) \circ K_1(\cdot)$ in 
direction $\tilde{W}_3$ we have to replace $\frac{\rand}{\rand x}$
by $\frac{\rand}{\rand y}$ in the above formula. Consequently,
from \eqref{eq:14}
\begin{align*}
\sgn \det(D_{\mathcal{Z}_{M}}(P_1(\cdot)\circ K_1(\cdot))\eval_{\tilde{\alpha}(x_0,y_0)}) 
= \deg_{loc}(\nabla k_1,\pi_M(w_0)).
\end{align*}
Thus we arrive at the following
\begin{lemma}
\label{l:k_1_2}
Let $\{\tilde{w}_i\where 1\le i \le n\}$ denote the set of critical points of $k_1$
in $M$. Then there is $r_0>0$ such that
for all $0<r\le r_0$ the set of critical orbits of 
$P_1(\cdot )\circ K_1(\cdot)$ is given by $\{S^1*\tilde{\alpha}_{i,r}\where 1\le i \le n\}$,
where 
\begin{align*}
\tilde{\alpha}_{i,r}= \pi_M(\sqrt{1+r^2}w_{i,r}+r\cos(2\pi t)v_{1,i,r}
+r\sin(2\pi t)v_{0,i,r}) \in \mathcal{Z}_{M}.  
\end{align*}
Moreover, we have for $1\le i\le n$
\begin{align*}
\pi_M(w_{i,r}) \to  \tilde{w}_i \text{ as }r \to 0^+,\\
\sgn \det(D_{\mathcal{Z}_{M}}(P_1(\cdot)\circ K_1(\cdot))\eval_{\tilde{\alpha}_{i,r}}) 
= \deg_{loc}(\nabla k_1,\tilde{w}_i).   
\end{align*}
\end{lemma}
\begin{proof}
From Lemma \ref{l:k_1_1} and the analysis of $H$ we may choose $\delta>0$ and $r_0>0$
such that the union $\cup_{i=1}^n B_{\delta}(\tilde{w}_i)$ is disjoint and
for every $i$ and $0<r\le r_0$ 
there is a unique $\pi_M(w_{i,r}) \in B_{\delta}(\tilde{w}_i)$ corresponding
to a critical orbit $S^1*\tilde{\alpha}_{i,r}$. Moreover, 
if $r \to 0^+$ we may shrink $\delta>0$, which yields together with the uniqueness
of $\pi_M(w_{i,r})$ the claimed asymptotic.     
\end{proof}

\begin{lemma}
\label{l:euler_charak_unperturbed}
Let $\mathscr{M}_A$ be the set of oriented Alexandrov embedded regular curves
in $H^{2,2}(S^1,M)$. There is $C_{k_0}>0$ such that for all $k_0\ge C_{k_0}$
we have
\begin{align*}
\chi_{S^1}(X_{k_0,g_0},\mathscr{M}_A)=-\chi(M),  
\end{align*}
where $\chi(M)$ denotes the Euler characteristic of $M$.
\end{lemma}
\begin{proof}
We choose a Morse function $k_1$ on $M$ with nondegenerate critical points
$\{\tilde{w}_i\where 1\le i \le n\}$.
From Lemma \ref{l:k_1_2} we obtain $C_{k_0}>0$ such that for all
$k_0\ge C_{k_0}$ the critical orbits of $P_1(\cdot)\circ K_1(\cdot)$
are given by $\{S^1*\tilde{\alpha}_{i,k_0}\where 1\le i \le n\}$ satisfying
\begin{align*}
\sgn \det(D_{\mathcal{Z}_{M}}(P_1(\cdot)\circ K_1(\cdot))\eval_{\tilde{\alpha}_{i,k_0}}) 
= \deg_{loc}(\nabla k_1,\tilde{w}_i).  
\end{align*}
We fix $k_0\ge C_{k_0}$.
By Lemma \ref{l:nondeg_zero} there is $\eps>0$ such that
for any $0<\eps<\eps_0$ and $1\le i\le n$
there is $\tilde{\gamma}_{i}(\eps) \in \Phi(\mathcal{U}_i)$ satisfying
\begin{align*}
X_{g_0,\eps}(\tilde{\gamma}_{i}(\eps))=0 \text{ and }
\tilde{\gamma}(\eps) \to \tilde{\alpha}_{i,k_0} \text{ as } \eps \to 0.  
\end{align*}
Moreover, $S^1*\tilde{\gamma}_{i}(\eps)$ is the unique critical orbit of $X_{g_0,\eps}$ 
in $\Phi(\mathcal{U}_i)$ and is nondegenerate with
\begin{align}
\label{eq:16}
\deg_{loc, S^{1}}(X_{g_0,\eps},S^1*\tilde{\gamma}(\eps)) &= -\deg_{loc}(\nabla k_1,\tilde{w}_i).  
\end{align}
To show that there is an open neighborhood
$\mathcal{U}$ of $\mathcal{Z}_M$ and $\eps_0>0$
such that for all $0<\eps<\eps_0$ the critical orbits
of $X_{g_0,\eps}$ in $\mathcal{U}$ are given exactly 
by $\{\tilde{\gamma}_i(\eps)\where 1 \le i \le n\}$
we argue by contradiction. Suppose there are $\eps_n \to 0^+$ and
a sequence $(\tilde{\alpha}_n)$ of zeros of $X_{g_0,\eps_n}$ that converges to
$\mathcal{Z}_M$ but $\tilde{\alpha}_n \notin \{\tilde{\gamma}_i(\eps)\where 1 \le i \le n\}$.
Up to a subsequence we may assume
\begin{align*}
\tilde{\alpha}_n \to \tilde{\alpha}_0 \in \mathcal{Z}_M
\end{align*}
as $n \to \infty$. For large $n$ we use the chart $\Phi$ around $\tilde{\alpha}_0$
given in \eqref{eq:15}. 
From the existence of a slice of the $S^{1}$-action (see \cite[Lem. 3.1]{arXiv:0808.4038})
we get sequences $\theta_n \in \rz/\zz$ and  
$V_n \in \langle \tilde{W}_1(\tilde{\alpha}_0)\rangle^\perp$
converging to $0$ such that
\begin{align*}
\theta_n*\tilde{\alpha}_n = \Phi(V_n).
\end{align*}
Note that from the $S^1$-invariance and by construction  
\begin{align*}
X_{g_0,\eps_n}(\theta_n*\tilde{\alpha}_n)=0 \text{ and }
X_{g_0,\eps_n}^{\Phi}(V_n)=0. 
\end{align*}
We consider the map
\begin{align*}
\Lambda:\:
\langle \tilde{W}_1(\tilde{\alpha}_0)\rangle^{\perp}=  
\langle \tilde{W}_2(\tilde{\alpha}_0),\tilde{W}_3(\tilde{\alpha}_0)\rangle 
\oplus R\big(D_{g_0}X_{g_0,0}^{\Phi}\eval_{0}\big)
\to \langle \tilde{W}_1(\tilde{\alpha}_0)\rangle^{\perp},  
\end{align*}
defined by
\begin{align*}
\Lambda(W,V) := W + D_{g_0}X_{g_0,0}^{\Phi}\eval_{W}(V).  
\end{align*}
From \eqref{eq:decomp_phi_w} the map $\Lambda$ is a diffeomorphism 
locally around $(0,0)$, hence 
we may decompose
\begin{align*}
V_n= \Phi^{-1}(\theta_n*\alpha_n)= W_n+U_n,  
\end{align*}
where $W_n \in \langle \tilde{W}_1(\alpha_0)\rangle^\perp$ and 
$U_n \in R(D_{g_0}X_{k_0,g_0}^\Phi\eval_{W_n})$ converge to $0$
as $n \to \infty$.
From the uniqueness part of Lemma \ref{l:implicit_function}, 
as $X_{g_0,\eps_n}^{\Phi}(W_n+U_n)=0$,
we get $U_n= U(\eps_n,W_n)$. By Lemma \ref{l:expansion_eps_0} we see that
\begin{align*}
P_1(\tilde{\alpha}_0)\circ K_1(\tilde{\alpha}_0)=0.  
\end{align*}
Consequently, $S^1*\alpha_0 \in \{S^1*\tilde{\alpha}_{i,k_0}\where 1\le i \le n\}$.
From the uniqueness part in Lemma \ref{l:nondeg_zero} we finally arrive at the contradiction
\begin{align*}
S^1*\tilde{\alpha}_n \in \{S^1*\tilde{\gamma}_{i}(\eps_n)\where 1 \le i \le n\}.  
\end{align*}
From the definition of the $S^1$-equivariant Poincar\'{e}-Hopf index, 
the classification of Alexandrov embedded zeros of $X_{k_0,g_0}$, and \eqref{eq:16}
there holds for small $\eps>0$ 
\begin{align*}
\chi_{S^1}(X_{k_0,g_0},\mathscr{M}_A) &= \chi_{S^1}(X_{k_0,g_0},\mathcal{U})
= \chi_{S^1}(X_{g_0,\eps},\mathcal{U})\\
&= \sum_{i=1}^n \deg_{loc, S^{1}}(X_{g_0,\eps},S^1*\tilde{\gamma}_i(\eps))\\
&= -\sum_{i=1}^n \deg_{loc}(\nabla k_1,\tilde{w}_i) = -\chi(M).  
\end{align*}
\end{proof}

\section{The apriori estimate}
\label{sec:apriori-estimate}
We fix a continuous family of metrics $\{g_t\where t\in [0,1]\}$ on $M$
and a continuous family of positive continuous function $\{k_t \where t\in [0,1]\}$ on $M$.
We assume that there is $K_0>0$, 
such that the Gaussian curvature $K_{g_t}$ of each metric $g_t$ on $M$ 
and the functions $\{k_t\}$ satisfy
\begin{align}
\label{eq:17}
K_{g_t} \ge -K_0,\\  
\label{eq:18}
k_{inf}:= \inf\{k_t(x)\where (x,t)\in M\times [0,1]\}>(K_0)^\frac12.  
\end{align}
We let $X_{t}$ be the vector field on $H^{2,2}(S^1,M)$ defined by
\begin{align*}
X_t:= X_{k_t,g_t}.  
\end{align*}
We denote by $\mathscr{M}_A\subset H^{2,2}(S^1,M)$ the set
\begin{align*}
\mathscr{M}_A:= \{\gamma \in H^{2,2}_{reg}(S^1,M) 
\where \gamma \text{ is prime and oriented Alexandrov embedded.}\}.
\end{align*}
We shall show that the set 
\begin{align*}
X^{-1}(0) := \{(\gamma,t) \in \mathscr{M}_A\times [0,1] \where X_t(\gamma)=0\}   
\end{align*}
is compact in $\mathscr{M}_A\times [0,1]$.
Fix $(\gamma,t)\in X^{-1}(0)$. Then there is an oriented immersion $F:\overline{B}\to M$
with $F\eval_{\rand B}=\gamma$. We denote by $F^*g_t$ the induced metric on $B$.
\begin{lemma}
\label{lem_soln_prescribed_gauss}
For any $(\gamma,t)\in X^{-1}(0)$
there is $\phi \in C^2(\overline{B},\rz)$ satisfying
\begin{align}
\label{eq:19}
-\laplace_{F^*g_t}\phi + K_{F^*g_t}+K_0 e^\phi = 0 \text{ in } B, \notag \\
\rand_{\nu} \phi = 0 \text{ on } \rand B,  
\end{align}
where $\nu$ denotes the unit normal oriented to the outside.\\
Moreover, there is $C_0>0$, which may be chosen independently of $(\gamma,t)\in X^{-1}(0)$,
such that
\begin{align*}
0\ge \phi \ge -C_0.  
\end{align*}
\end{lemma}
\begin{proof}
To show the existence of a solution $\phi$ we use the method of upper and lower solutions 
(see also \cite{MR0343205}).
The function $\phi_+ \equiv 0$ satisfies
\begin{align*}
-\laplace_{F^*g_t}\phi_+ + K_{F^*g_t}+K_0 e^{\phi_+} = K_{F^*g_t}+K_0 \ge 0,   
\end{align*}
from \eqref{eq:17} and the fact that $F$ is a local isometry. Hence, $\phi_+$
is a supersolution of \eqref{eq:19}. To find a subsolution, we let $\phi_1 \in C^\infty(M,\rz)$
be defined as the solution to the linear equation
\begin{align*}
-\laplace_{g_t}\phi_1 + K_{g_t}-2\pi \chi(M) \text{vol}(M,g_t)^{-1}= 0 \text{ in }M,\;
\int_M \phi_1 dg_t =0.
\end{align*}
By standard elliptic estimates using a Green's function on $(M,g_t)$ 
(see \cite[Thm 4.13]{MR1636569}) we have
\begin{align*}
\sup_{M}|\phi_1| &\le C(g_t) 
\big(\sup_{M}|K_{g_t}|- 2\pi \chi(M) \text{vol}(M,g_t)^{-1}\big)\\
&\le C_1,
\end{align*}
because $\{g_t\where t\in [0,1]\}$ is a compact set of smooth metrics.
We may choose $C_2>1$ such that we have for all $t\in [0,1]$ 
\begin{align*}
-C_2 \le \ln\big(-2\pi \chi(M)K_0 \text{vol}(M,g_t)^{-1}\big).
\end{align*}
Since $F$ is a local isometry, there holds
\begin{align*}
\laplace_{F^*g_t} (\phi_1 \circ F)= (\laplace_{g_t} \phi_1) \circ F.   
\end{align*}
We define $\phi_- \in C^2(B,\rz)$ by
\begin{align*}
\phi_-:= \phi_1\circ F -C_1-C_2  
\end{align*}
and get
\begin{align*}
-\laplace_{F^*g_t}\phi_- + K_{F^*g_t}+K_0 e^{\phi_-} =
2\pi \chi(M) \text{vol}(M,g_t)^{-1} +K_0 e^{\phi_1 \circ F-C_1-C_2}
\le 0. 
\end{align*}
Hence $\phi_-$ is a subsolution of \eqref{eq:19} satisfying
\begin{align*}
-C_0 := -(2C_1+C_2) \le \phi_-  < \phi_+. 
\end{align*}
Using a version of the method of upper and lower solutions
given in \cite{MR2316146} we find a solution $\phi$ to \eqref{eq:19}
satisfying $\phi_- \le \phi \le \phi_+$.
\end{proof}
We consider $B$ equipped with the metric  $h_t:= e^\phi F^* g_t$.
Then the Gaussian curvature $K_{h_t}$ and the geodesic
curvature $k_{h_t}$ of $\rand B$ with respect to $(B,h_t)$ are given by 
(see \cite[Sec 5.8.2]{MR1636569})
\begin{align*}
K_{h_t} \equiv -K_0 \text { and } k_{h_t} = k_{F^*g_t} e^{-\frac{\phi}{2}}.   
\end{align*}
Consequently, since $0\le \phi$,
\begin{align*}
\inf_{\rand B} k_{h_t} \ge \inf_{\rand B} k_{F^*g_t} 
\ge k_{\inf}.
\end{align*}
The Gauss-Bonnet formula applied to $(B,h_t)$ gives
\begin{align*}
2\pi &= -\int_{B} K_0 \,dh_t +
\int_{\rand B}k_{h_t} \, dS_{h_t}
&\ge
-K_0 A(B,h_t)+k_{inf} L(\rand B, h_t),  
\end{align*}
where $A(B,h_t)$ denotes the area of $B$ and $L(\rand B, h_t)$ the length of $\rand B$
with respect to $h_t$.
The isoperimetric inequality (see \cite[Thm 4.3]{MR0500557}) yields
\begin{align*}
L(\rand B, h_t)^2 \ge 4\pi A(B,h_t) + K_0 A(B,h_t)^2 \ge K_0 A(B,h_t)^2.    
\end{align*}
Thus we arrive at
\begin{align*}
2\pi \ge -(K_0)^{\frac12} L(\rand B, h_t) + k_{inf} L(\rand B, h_t).    
\end{align*}
This yields
\begin{align*}
L(\gamma,g_t) &= L(\rand B,F^* g_t) \le e^{C_0} L(\rand B,h_t)
\le e^{C_0} \frac{2\pi}{k_{inf}- (K_0)^{\frac12}}.     
\end{align*}
Using again the Gauss-Bonnet formula we see
\begin{align*}
2\pi &= -\int_{B} K_0 \,dh_t +
\int_{\rand B}k_{h_t} \, dS_{h_t}\\
&\le e^{\frac{C_0}{2}}\big(\sup\{k_t(x)\where (x,t)\in M\times [0,1]\}\big) 
L(\rand B,h_t)\\
&\le e^{2C_0}\big(\sup\{k_t(x)\where (x,t)\in M\times [0,1]\}\big)
L(\gamma,g_t).  
\end{align*}
Consequently, there is $C>0$, such that
\begin{align}
\label{eq:21}
C \le L(\gamma,g_t) \le C^{-1},  
\end{align}
for all $(\gamma,t) \in X^{-1}(0)$.\\
Fix a sequence $(\gamma_n,t_n)_{n\in \nz}$ in $X^{-1}(0)$. 
As a solution each $\gamma_n$ is parameterized proportional
to its arc-length. From \eqref{eq:21},
$(\gamma_n)$ is uniformly bounded in $C^1(S^1,M)$. Using the equation
\eqref{eq:1} we obtain a uniform bound of $(\gamma_n)$
in $C^3(S^1,M)$, such that we may extract a subsequence, still denoted
by $(\gamma_n,t_n)_{n \in \nz}$, which
converges in $C^2(S^1,M)\times [0,1]$ to $(\gamma_0,t_0)$.
The convergence in $C^2(S^1,M)$ and the lower bound in \eqref{eq:21} imply that
$X_{t_0}(\gamma_0)=0$ and that $\gamma_0$ is an immersion.
By Lemma \ref{lem:basis_alexandrov}
the curve $\gamma_0$ is oriented Alexandrov embedded and hence
$(\gamma_0,t_0)\in X^{-1}(0)$. This shows that
\begin{lemma}
\label{l:apriori-estimate}
Under the assumptions \eqref{eq:17} and \eqref{eq:18}
the set $X^{-1}(0)$ is compact. 
\end{lemma}

\section{Existence results}
\label{sec:existence}
We give the proof of our main existence result.
\begin{proof}[{Proof of Theorem \ref{thm_existence}}]
From the uniformization theorem 
$(M,g)$
is isometric to $(\hz/\Gamma,e^\phi g_0)$,
where $\Gamma \subset O(2,1)_+$ is a group
of isometries acting
freely and properly discontinuously and $\phi$ is a function in
$C^\infty(\hz/\Gamma,\rz)$.
Due to the invariance of \eqref{eq:1} under isometries we may assume
without loss of generality that
\begin{align*}
(M,g) = (\hz/\Gamma,e^\phi g_0). 
\end{align*}
We consider the family of metrics $\{g_t:= e^{t\phi} g_{0} \where t\in [0,1]\}$
and choose a large constant $k_0>>1$, such that
\begin{align*}
k_0 > \big(-\inf\{K_{g_t}(x) \where (x,t) \in M\times [0,1]\}\big)^\frac12
+\inf_{M}k + C_{k_0},  
\end{align*}
where $K_{g_t}$ denotes the Gaussian curvature of the metric $g_t$ given by
\begin{align*}
K_{g_t} &= e^{-t\phi}\big(-t\laplace_{g_0}(\phi)+2\big).
\end{align*}
From Lemma \ref{l:apriori-estimate} the homotopy 
\begin{align*}
[0,1]\ni t\mapsto X_{k_0,g_t}  
\end{align*}
is $(\mathscr{M}_A,g_t,S^1)$-admissible. By Lemma \ref{l:euler_charak_unperturbed}
and the homotopy invariance of the $S^{1}$-equivariant Poincar\'{e}-Hopf index
we obtain 
\begin{align*}
-\chi(M)= \chi_{S^1}(X_{k_0,g_0},\mathscr{M}_A)=\chi_{S^1}(X_{k_0,g},\mathscr{M}_A).  
\end{align*}
For $t\in [0,1]$ we define $k_t \in C^\infty(M,\rz)$ by
\begin{align*}
k_t(x) := (1-t)k_0 + t k(x).  
\end{align*}
Then 
\begin{align*}
\inf\{k_t(x)\where (x,t)\in M\times [0,1]\}= \inf_{M}k > 
\big(-\inf_{M}K_{g}\big)^\frac12.  
\end{align*}
From Lemma \ref{l:apriori-estimate} the homotopy 
\begin{align*}
[0,1]\ni t\mapsto X_{k_t,g}  
\end{align*}
is $(\mathscr{M}_A,g,S^1)$-admissible and there holds
\begin{align*}
\chi_{S^1}(X_{k,g},\mathscr{M}_A)=\chi_{S^1}(X_{k_0,g},\mathscr{M}_A)= -\chi(M).
\end{align*}
This gives the claim.
\end{proof}

\bibliographystyle{plain}
\bibliography{geodesic_curves}

\end{document}